\documentclass[11pt,reqno,a4paper]{amsart}

\usepackage{fbb}
\usepackage{lmodern}
\usepackage[T1]{fontenc}
\usepackage[utf8]{inputenc}
\usepackage{amsmath,amssymb,amsthm,mathtools}

\usepackage[top=2.5cm,bottom=2.4cm,left=2.8cm,right=2.8cm,headsep=0.2in,headheight=14pt,footskip=24pt]{geometry}
\usepackage{fancyhdr}
\usepackage[shortlabels]{enumitem}
\usepackage[expansion=false]{microtype}
\usepackage{xcolor}
\definecolor{linkteal}{rgb}{0.0,0.42,0.42}
\usepackage[colorlinks=true,linkcolor=linkteal,citecolor=linkteal,urlcolor=linkteal]{hyperref}
\hypersetup{
  pdftitle={Exact versus monotone isoperimetric profiles of finitely generated groups},
  pdfauthor={Mayukh Mukherjee},
  pdfsubject={Exact and monotone isoperimetric profiles, wreath products, and nested Folner sets},
  pdfkeywords={isoperimetric profile, Folner sequence, Sierpinski carpet group, Property FW, polycyclic group, Grigorchuk group, Schreier graph, wreath product, amenable group}
}
\usepackage[nameinlink,noabbrev]{cleveref}

\allowdisplaybreaks
\usepackage{aliascnt}

\newtheorem{theorem}{Theorem}[section]
\newaliascnt{proposition}{theorem}
\newtheorem{proposition}[proposition]{Proposition}
\aliascntresetthe{proposition}
\newaliascnt{lemma}{theorem}
\newtheorem{lemma}[lemma]{Lemma}
\aliascntresetthe{lemma}
\newaliascnt{corollary}{theorem}
\newtheorem{corollary}[corollary]{Corollary}
\aliascntresetthe{corollary}
\newaliascnt{conjecture}{theorem}

\aliascntresetthe{conjecture}
\newtheorem{theoremalph}{Theorem}

\theoremstyle{definition}
\newaliascnt{definition}{theorem}
\newtheorem{definition}[definition]{Definition}
\aliascntresetthe{definition}
\newaliascnt{question}{theorem}
\newtheorem{question}[question]{Question}
\aliascntresetthe{question}
\theoremstyle{remark}
\newaliascnt{remark}{theorem}
\newtheorem{remark}[remark]{Remark}
\aliascntresetthe{remark}
\newaliascnt{warning}{theorem}

\aliascntresetthe{warning}

\crefname{theorem}{Theorem}{Theorems}
\Crefname{theorem}{Theorem}{Theorems}
\crefname{proposition}{Proposition}{Propositions}
\Crefname{proposition}{Proposition}{Propositions}
\crefname{lemma}{Lemma}{Lemmas}
\Crefname{lemma}{Lemma}{Lemmas}
\crefname{corollary}{Corollary}{Corollaries}
\Crefname{corollary}{Corollary}{Corollaries}
\crefname{conjecture}{Conjecture}{Conjectures}
\Crefname{conjecture}{Conjecture}{Conjectures}
\crefname{definition}{Definition}{Definitions}
\Crefname{definition}{Definition}{Definitions}
\crefname{question}{Question}{Questions}
\Crefname{question}{Question}{Questions}
\crefname{remark}{Remark}{Remarks}
\Crefname{remark}{Remark}{Remarks}
\crefname{warning}{Warning}{Warnings}
\Crefname{warning}{Warning}{Warnings}
\crefname{theoremalph}{Theorem}{Theorems}
\Crefname{theoremalph}{Theorem}{Theorems}
\crefname{section}{Section}{Sections}
\Crefname{section}{Section}{Sections}
\crefname{subsection}{Section}{Sections}
\Crefname{subsection}{Section}{Sections}
\crefname{equation}{Equation}{Equations}
\Crefname{equation}{Equation}{Equations}

\newcommand{\GammaG}{\Gamma}
\newcommand{\cS}{\mathcal S}
\newcommand{\cT}{\mathcal T}
\newcommand{\N}{\mathbb N}
\newcommand{\Z}{\mathbb Z}
\newcommand{\R}{\mathbb R}
\newcommand{\card}[1]{\lvert #1\rvert}
\newcommand{\bd}{\partial_{\cS}}
\newcommand{\edgebd}{B_{\cS}}
\newcommand{\Iex}{I}
\newcommand{\Iup}{I_{\uparrow}}
\newcommand{\IexE}{I_E}
\newcommand{\IupE}{I_{E,\uparrow}}

\newcommand{\dist}{\operatorname{dist}_{\cS}}
\newcommand{\supp}{\operatorname{supp}}

\newcommand{\one}{\mathbf 1}
\newcommand{\e}{\mathrm e}
\newcommand{\cut}{B}

\title[Exact versus monotone profiles]
{Exact versus monotone isoperimetric profiles of finitely generated groups}
\author{Mayukh Mukherjee}
\date{}

\begin{document}

\begin{abstract}
Gromov asked whether the exact isoperimetric profile of an infinite finitely generated group is comparable, at the same cardinality, to its greatest nondecreasing minorant. We compare finite sets with their translates and construct orders for which the first $r$ vertices have boundary bounded by a fixed multiple of that of the whole set. Such orders follow from profile summability or from rays in Schreier graphs with a uniform bound on the number using each edge.

Among infinite finitely generated groups, we obtain comparison and nested exhaustions for elementary amenable and branch groups, direct products of at least two infinite groups, restricted wreath products with two nontrivial factors, and certain simple topological full groups, including amenable derived groups of interval exchanges. The exhausting $n$-point sets have boundary comparable to the monotone profile at $n$. The ray condition is quasi-isometry invariant and holds for the Sierpi\'nski-carpet group. For polycyclic groups of exponential growth, we refine Pittet's construction to obtain explicit nested sets with $|F_n|=n$ and $|\partial F_n|\asymp n/\log n$, answering his question with exact rather than comparable cardinalities.

We also prove that the normalized $L^1$ profile satisfies $j(v)\le C(v/u)j(u)$ for $1\le u\le v$ on every infinite connected bounded-degree vertex-transitive graph. This yields the separation bound asked for by Gournay--Le Coz and Hume's Dirichlet--Neumann comparison at $p=1$. For groups of exponential growth, the exact-to-monotone comparison ratio is $O(\log\log n)$. We answer the first part of Correia's Question~1.3 positively by proving the required dilation regularity, and the second negatively by constructing elementary amenable examples whose inverse compositions fail this regularity for every choice of equivalent representatives.
\end{abstract}

\maketitle
\thispagestyle{fancy}
\enlargethispage{5pt}

\section{Introduction}
Let $G$ be an infinite finitely generated group with a finite symmetric generating set $\cS$. Its right Cayley graph joins $x$ to $xs$ for $s\in\cS$. For finite $F\subset G$, the outer vertex boundary $\partial F=F\cS\setminus F$ consists of the vertices outside $F$ adjacent to $F$. Define
\[
 I(n)=\min_{|F|=n}|\partial F|,
 \qquad I_\uparrow(n)=\inf_{m\ge n}I(m).
\]
Thus $I(n)$ asks for the smallest boundary of a set with exactly $n$ vertices, whereas $I_\uparrow(n)$ allows any cardinality at least $n$. The function $I_\uparrow$ is the largest nondecreasing function bounded above by $I$. The exact profile can decrease, as the example on $\Z\times(\Z/2\Z)$ in \Cref{sec:conventions} shows. Gromov asked whether
\begin{equation}\label{eq:gromov-question}
 \sup_{n\ge1}\frac{I(n)}{I_\uparrow(n)}<\infty
\end{equation}
for every infinite finitely generated group \cite[p.~501]{Gromov2008}. In other words, can one find an $n$-point set whose boundary is at most a fixed multiple of the best boundary among sets of at least $n$ vertices? The comparison is at the same $n$.

We also seek a nested exhaustion: sets that grow one vertex at a time and eventually cover $G$, with
\[
 W_n\subset W_{n+1},\qquad \bigcup_{n\ge n_0}W_n=G,
 \qquad |W_n|=n,\qquad |\partial W_n|\le C I_\uparrow(n).
\]
An amenable group has nonempty finite sets with arbitrarily small boundary-to-cardinality ratio. For amenable $G$, $I_\uparrow(n)/n\to0$. Thus our exhaustions are F{\o}lner sequences: their boundary-to-cardinality ratios tend to zero.

It is useful to count crossing edges as well as boundary vertices. Put $\Delta=|\cS|$ and define
\[
\begin{gathered}
 B_{\cS}(F)=\#\{(x,s)\in F\times\cS:xs\notin F\},\qquad
 I_E(n)=\min_{|F|=n}B_{\cS}(F),\\
 a(n)=I_{E,\uparrow}(n)=\inf_{m\ge n}I_E(m),\qquad
 j(n)=\max_{0<|F|\le n}\frac{|F|}{B_{\cS}(F)}.
\end{gathered}
\]
Write $B=B_{\cS}$. Symmetry of $\cS$ makes this count each edge from $F$ to its complement once. The value $a(n)$ is the smallest edge boundary among sets of at least $n$ vertices. The value $j(n)$ is the largest cardinality-to-boundary ratio among nonempty sets of at most $n$ vertices; a larger value means that a smaller relative boundary is available. For real $v\ge1$, set $j(v)=j(\lfloor v\rfloor)$. Each outer boundary vertex is incident to between one and $\Delta$ crossing edges, so vertex and edge profiles differ by at most the factor $\Delta$.

We seek an ordering of every finite nonempty set $A$ for which all initial segments have controlled boundary. The initial segments of an order $A=\{x_1,\ldots,x_m\}$ are the first $r$ vertices. We want
\[
 B(\{x_1,\ldots,x_r\})\le K B(A)\qquad(1\le r\le m),
\]
with $K$ independent of $A$ and $r$. If such orders exist, taking the first $n$ vertices of a set attaining $a(n)$ gives $I_E(n)\le K a(n)$. The same orders also give nested exhaustions. We obtain such orders under either a bound on the profile or a condition on paths in a graph of a group action.

We first prove general profile estimates without assuming these orders exist. The argument works on infinite connected vertex-transitive graphs of finite degree: a graph automorphism can send any vertex to any other. On such a graph, $B(F)$ counts edges from $F$ to its complement, and the profiles use the same formulas. Choose a set of least cardinality among those with at least $n$ vertices and boundary $a(n)$. This is an edge $n$-atom, following \cite[Lemma~1 and p.~142]{HLST2000}. Comparing it with its automorphic images either gives $I_E(n)\le\Delta a(n)$ or produces $Q$ with $n/2<|Q|<n$ and $B(Q)\le a(n)$. In the second case, put $r=n-|Q|<n/2$ and add a disjoint image of an $r$-point set with boundary $I_E(r)$. The union has boundary at most $a(n)+I_E(r)$. Iterating this inequality at the smaller cardinality gives the sum in the next theorem.

\begin{theoremalph}[Profile estimates]\label{thm:main-universal}
Let $X$ be an infinite connected vertex-transitive graph of finite degree $\Delta$, with edge profiles $I_E,a,j$ defined as above. For integers $1\le u\le v$ and $n\ge1$,
\[
 j(v)\le\Delta\frac vu j(u),\qquad
 \frac{n}{\Delta j(n)}\le a(n)\le\frac{2n}{j(n)}.
\]
With $a(0)=0$, one also has
\[
 I_E(n)\le\Delta\sum_{k\ge0}a(\lfloor n/2^k\rfloor).
\]
\end{theoremalph}

The first inequality limits the improvement in $j$ when larger sets are allowed. The next two identify $a(n)$, up to fixed factors, with $n/j(n)$. The last estimate gives $I_E(n)=O(a(n))$ if there is one constant $D$ such that
\begin{equation}\label{eq:intro-dyadic-bound}
 \sum_{k\ge0}a(\lfloor n/2^k\rfloor)\le D a(n)\qquad(n\ge1).
\end{equation}
Each sum is finite: its terms vanish once $2^k>n$. The condition requires a uniform bound on the sum of the values at successively halved cardinalities, relative to the original value $a(n)$. A sufficient condition is that, for some fixed $q>1$ and $0<\theta<1$,
\[
 a(\lfloor n/q\rfloor)\le\theta a(n)
\]
for all sufficiently large $n$. Dividing cardinality by $q$ then reduces the profile by a fixed fraction. For groups of exponential growth, meaning that Cayley balls grow exponentially with their radius, the general estimate gives $I_E(n)/a(n)=O(\log\log(e^e+n))$; see \Cref{cor:scaling,cor:growth}.

For Cayley graphs, \eqref{eq:intro-dyadic-bound} gives the required orders by repeated subdivision. Divide $A$ into two parts, each with at most three quarters of its vertices, and repeat until singletons remain. List the first part before the second at every division. Left translation $x\mapsto gx$ preserves the right Cayley graph. Averaging intersections with left translates gives a division of an $m$-point set with at most a fixed multiple of $a(m)$ edges between the parts, following \cite[\S\S2.2--2.3]{FreedmanGuthMurphy2015}. For an initial segment ending at $x_r$, only the successive parts containing $x_r$ can be partly listed. Every internal edge leaving the initial segment is cut at one of these divisions. The sizes of these successive parts decrease geometrically, so \eqref{eq:intro-dyadic-bound} controls these edges by a fixed multiple of $a(|A|)\le B(A)$. Edges leaving $A$ contribute at most another $B(A)$ (\Cref{thm:dini-orderings}).

To obtain a nested exhaustion when $a$ is unbounded, let $r_k$ be the first cardinality where $a(r_k)\ge2^k$. Choose $A_k$ with $|A_k|\ge r_k$ and $B(A_k)=a(r_k)$. Since $a(n+1)\le a(n)+\Delta$, these boundaries lie between $2^k$ and $2^k+\Delta$, so their sum is controlled by the last one. Add disjoint left translates of these sets one vertex at a time in their chosen orders. While the next set is added, the sum of all completed boundary sizes and the boundary of the new initial segment is bounded by a fixed multiple of $a(r_k)$. The current cardinality $n$ is at least $|A_k|\ge r_k$, so $a(r_k)\le a(n)$. Also insert the first missing element of a fixed enumeration at each stage, so the sets exhaust $G$. If $a$ is bounded, $G$ has finitely many cosets of an infinite cyclic subgroup. Expanding intervals in these cosets gives an exhaustion with bounded boundary. In both cases, \Cref{thm:record-stacking} gives $|W_n|=n$ and $B(W_n)\le C a(n)$ for all sufficiently large $n$.

For a direct product $G=G_1\times\cdots\times G_k$ of $k\ge2$ infinite finitely generated groups, the classical product-profile formula combines the factor profiles. Together with \Cref{thm:main-universal}, it yields
\[
 a_G(v)\ge c(v/u)^{1-1/k}a_G(u)\qquad(\text{for sufficiently large integers }v\ge u).
\]
The exponent is positive. Dividing the argument by a sufficiently large fixed factor therefore reduces $a_G$ by a fixed fraction. Thus \eqref{eq:intro-dyadic-bound} gives comparison and nested exhaustions (\Cref{prop:mixed-products}).

Here $f(n)\asymp g(n)$ means comparison in both directions by fixed positive factors at the same $n$, for all sufficiently large $n$.

Let $\Gamma,H$ be nontrivial finitely generated groups. Write $W=\Gamma\wr H$ for their restricted wreath product. In $W$, the group $H$ indexes copies of $\Gamma$ and shifts their positions. An element consists of an element of $H$ and a function $H\to\Gamma$ equal to the identity outside a finite set. For amenable factors and infinite $H$, Erschler's wreath-product estimate \cite{Erschler2003} requires $j_H(Cv)=O_C(j_H(v))$ for each fixed $C>0$ \cite[Theorem~2.19]{CorreiaDumoncel2026}. This means that multiplying the argument by $C$ changes the profile by at most a factor depending on $C$, not on $v$. \Cref{thm:main-universal} supplies the bound. The wreath estimate gives $a_W(qn)\ge2a_W(n)$ for some fixed $q>1$ and all large $n$, hence \eqref{eq:intro-dyadic-bound}. If either factor is nonamenable, $a_W(n)\asymp n$ gives the same bound directly. Infinite wreath products with finite $H$ are covered by the subgroup argument below.

\begin{theoremalph}[Wreath products]\label{thm:main-wreath}
Every infinite restricted wreath product $W=\Gamma\wr H$ of two nontrivial finitely generated groups satisfies \eqref{eq:gromov-question} for every finite generating set. It has a nested exhaustion $(W_n)_{n\ge n_0}$ with $|W_n|=n$ and $B(W_n)\le C a_W(n)$, where $a_W$ is its monotone edge profile. In the amenable case, this exhaustion is F{\o}lner. For finite nontrivial $\Gamma$ and infinite amenable $H$,
\[
 I_E^W(n)\asymp a_W(n)\asymp\frac{n}{j_H(\lfloor\log n\rfloor)}
\]
for all sufficiently large $n$. The constants may depend on the generating set.
\end{theoremalph}

Both comparison and nested exhaustion pass from a finite-index subgroup, one with finitely many cosets, to the whole group. If $H$ is finite and nontrivial and $\Gamma$ is infinite, the subgroup $\Gamma^H$ has finite index in $\Gamma\wr H$ and has at least two infinite direct factors. Finitely generated branch groups also have finite-index direct products of at least two infinite finitely generated groups, acting in disjoint subtrees of a rooted tree. Thus the same argument applies (Corollaries~\ref{cor:dini-nesting}, \ref{cor:finite-base-wreath}, and~\ref{cor:branch-groups}).

The bound \eqref{eq:intro-dyadic-bound} is not necessary for these orders. On $\Z$ with the usual generators, $a(n)=2$ and the finite sum grows like $\log n$. Yet listing any finite subset increasingly gives no larger boundary for any initial segment. We therefore examine directly the new boundary edges when a set is shortened.

Let $Y=\{y_1,\ldots,y_m\}$ be an ordered finite set in an infinite connected simple graph, with no loops or multiple edges and degree at most $\Delta$. Write $B$ for edge boundary and $y_i\sim y_j$ for adjacency. Put
\[
 P_r=\{y_1,\ldots,y_r\},\qquad
 D_r(Y)=\{y_r\}\cup\{y_i:i<r,\ y_i\sim y_j\text{ for some }j\ge r\}.
\]
Thus $D_r(Y)$ contains $y_r$ and the earlier vertices adjacent to a vertex at position $r$ or later. Every edge leaving $P_r$ either leaves $Y$ or joins $P_r$ to $Y\setminus P_r$. The first kind contributes at most $B(Y)$ edges; the second has an endpoint in $D_r(Y)$. Hence
\[
 B(P_r)\le B(Y)+\Delta|D_r(Y)|.
\]

A ray is a one-way infinite path without repeated vertices. Suppose one starts at each vertex of $D_r(Y)$ and each edge belongs to at most $c$ rays. Counting the first edge on which each ray leaves the finite set $Y$ gives
\[
 |D_r(Y)|\le cB(Y),\qquad B(P_r)\le(1+\Delta c)B(Y).
\]
We require one ordering of each finite nonempty $Y$ for which these rays exist for every $r$. The same integer $c\ge1$ must work for all sets and positions, although the rays may change (\Cref{def:condition}).

A transitive right action of $G$ on $X$ means that any point can be sent to any other. Its $\cS$-Schreier graph joins $x$ to $xs$ for $s\in\cS$; the simple graph removes loops and identifies parallel edges. For fixed $x_0\in X$, the map $g\mapsto x_0g$ projects the Cayley graph to it. Let $Y$ be the image of a finite set $F\subset G$, and choose its ordering from the condition on rays. Then list together the elements of $F$ with each image. These form a fiber within $F$. An initial segment in $F$ may end partway through a fiber, with boundary edges whose projection stays at the current image vertex $y_r$. Edges from earlier fibers to the unlisted part of this fiber also contribute. This is why $D_r(Y)$ includes $y_r$ and earlier vertices adjacent to it. Following the same generator labels lifts the rays to the Cayley graph. A fixed Cayley edge occurs in at most one lift of each chosen labelled ray, preserving the bound $c$ (\Cref{lem:lift}).

\begin{theoremalph}[Rays and initial segments]\label{thm:main-rays}
Let $G$ be infinite and finitely generated, let $\cS$ be a finite symmetric generating set with $e\notin\cS$, and put $\Delta=|\cS|$, $B=B_{\cS}$ and $a(n)=I_{E,\uparrow}(n)$. Suppose $G$ acts transitively on the right on an infinite set $X$ whose simple $\cS$-Schreier graph satisfies the condition on rays just stated, with constant $c$. Every finite nonempty $F\subset G$ admits an ordering for which
\[
 B(P)\le(1+\Delta c)B(F)\qquad\text{for every initial segment }P.
\]
Consequently $I_E(n)\le(1+\Delta c)a(n)$ for every $n\ge1$. There are $C_0=C_0(G,\cS,c)<\infty$, $n_0$, and a nested exhaustion $(W_n)_{n\ge n_0}$ of $G$ such that $|W_n|=n$ and $B(W_n)\le C_0a(n)$. If $G$ is amenable, these sets form a F{\o}lner sequence.
\end{theoremalph}

We prove \Cref{thm:main-rays} in \Cref{prop:linked-shelling,thm:record-stacking}. The condition on rays is invariant under quasi-isometry among bounded-degree graphs (\Cref{thm:transport}). A quasi-isometry is a map that compares distances up to fixed multiplicative and additive constants, with every target vertex within a bounded distance of its image; the constant $c$ may change.

Each exhaustion constructed here fixes an enumeration along which the values of $|f|$ can be placed in nonincreasing order. For every finitely supported real or complex $f$, this rearrangement increases the sum of absolute differences across edges by at most a fixed factor. This sum is the $L^1$ gradient seminorm. The coarea proof integrates the boundaries of the sets $\{|f|>t\}$, following Steinerberger \cite{Steinerberger2024}; see \Cref{cor:lone-rearrangement}.

One way to verify the condition on rays is to assign integer heights $h:X\to\Z$ to an infinite transitive action set. Suppose at most $W$ vertices have each height and every generator changes height by at most $L$, with fixed finite integers $W,L$. Order each finite $Y$ by increasing height. Then $D_r(Y)$ lies between heights $h(y_r)-L$ and $h(y_r)$, so it has at most $W(L+1)$ vertices. Each starts a ray because the graph is infinite, connected, and of bounded degree. Choosing one per vertex gives at most $W(L+1)$ rays in total, proving the condition with $c=W(L+1)$ and hence comparison and nested exhaustions (\Cref{thm:schreier-line-ratio,lem:height-fibres}).

A group is virtually indicable if a finite-index subgroup has a nonzero homomorphism to $\Z$. The subgroup acts on its infinite image in $\Z$, with at most one point at each height; the conclusions pass to the whole group. Every infinite finitely generated elementary amenable group is virtually indicable. Elementary amenable groups are built from finite and abelian groups by subgroups, quotients, extensions, and directed unions. The Grigorchuk family also has the required height actions, using Schreier graphs that are lines.

For a homeomorphism $T$ of a Cantor space, its topological full group consists of homeomorphisms that locally agree with powers of $T$. A Cantor space is homeomorphic to the usual Cantor set, and $T$ is minimal when every orbit is dense. The height argument applies to every infinite finitely generated simple subgroup of this full group when $T$ is minimal: the subgroup has an infinite orbit within a $T$-orbit identified with $\Z$, and each generator moves integers by a bounded amount. Examples include commutator subgroups of full groups of infinite minimal subshifts, closed shift-invariant spaces of bi-infinite sequences over a finite alphabet. They also include Nekrashevych's simple groups whose elements have finite order and whose Cayley balls grow faster than every polynomial but slower than exponentially. See \Cref{cor:virtually-indicable,cor:solvable,cor:grigorchuk-ratio,cor:full-groups}.

Orders of finite carpet graphs establish the condition on rays for the Sierpi\'nski-carpet group of Matte Bon--Nekrashevych--Zheng. They give paths to increasingly distant outer boundaries, with a uniform bound on the number using each edge. Limits of these paths give the required rays (\Cref{thmB}). The branch-group result also gives comparison and nested F{\o}lner exhaustions with boundary controlled by its monotone edge profile. The carpet group is the first infinite finitely generated amenable group with Property FW \cite{MatteBonNekrashevychZheng2026}. This means that every infinite connected Schreier graph is one-ended \cite[Proposition~7.I.3(vii)]{Cornulier2016}: deleting any finite set of vertices leaves exactly one infinite component. We show that Property FW excludes the height actions above, even for quotients of finite-index subgroups (\Cref{lem:fw-height}). The stronger conclusion about rays is used in the next application.

Let an infinite finitely generated group $H$ act freely by homeomorphisms on a Cantor space $C$: no nonidentity element fixes a point. Its full group consists of homeomorphisms agreeing with an element of $H$ on each piece of a finite clopen partition, where clopen means open and closed. Choose a clopen set $U$ and $h_1,h_2,h_3\in H$ with pairwise disjoint translates $h_1U,h_2U,h_3U$. A local three-cycle sends $h_1u$ to $h_2u$, then to $h_3u$, then back to $h_1u$ for each $u\in U$, and fixes the complement. These maps generate $\mathsf A(H,C)$, the alternating full group. Suppose the Cayley graph of $H$ satisfies the condition on rays and $\Lambda=\mathsf A(H,C)$ is finitely generated. Each $H$-orbit is one $\Lambda$-orbit, and its $\Lambda$-Schreier graph has distances comparable to those in the Cayley graph of $H$. Thus every finite generating set of $\Lambda$ gives comparison and a nested exhaustion with boundary controlled by its monotone edge profile (\Cref{thm:fullgroup}). Amenability of $\Lambda$ is needed only for the F{\o}lner conclusion.

Interval exchanges rearrange finitely many intervals by translations. The preceding result includes their simple amenable commutator subgroups associated with two irrational rotations generating a rotation group isomorphic to $\Z^2$. Neither these groups nor nontrivial quotients of their finite-index subgroups admit the height actions above (\Cref{cor:iet,prop:noheight}).

We also give an explicit construction for polycyclic groups of exponential growth. A polycyclic group has a finite chain of subgroups from the trivial group to $G$, each subgroup normal in the next, with cyclic successive quotients. Pittet's finite sets $\Omega_k$ satisfy \cite[p.~683]{Pittet1995}
\[
 \frac{|\partial\Omega_k|}{|\Omega_k|}\le\frac{C_1}{k},
 \qquad \log|\Omega_k|\le C_2k.
\]
Writing $m=|\Omega_k|$, these estimates give $|\partial\Omega_k|=O(m/\log m)$. We refine these sets so the successive differences are small enough to insert one vertex at a time while preserving the boundary estimate. This gives an explicit nested exhaustion with $|F_n|=n$ and $|\partial F_n|\asymp n/\log n$ for all sufficiently large $n$ (\Cref{thmC}). Abstract existence already follows from the height argument and Pittet's monotone-profile estimate.

On every infinite connected vertex-transitive graph $X$ of finite degree, the estimate for $j$ also bounds cuts and functions on finite subgraphs. For $A\subset X$ with $2\le|A|\le n$, some nonempty $U\subset A$ with $|U|\le|A|/2$ satisfies
\[
 |A|\frac{|E_X(U,A\setminus U)|}{|U|}\le C\frac n{j(n)},
\]
where $E_X(U,A\setminus U)$ is the set of edges joining $U$ to $A\setminus U$. This is the separation-profile bound asked for by Gournay--Le Coz in Question~7.3(1) \cite{GournayLeCoz2023}.

For functions on a finite vertex set $A$, retain all graph edges between its vertices; this gives the induced subgraph. Hume's $L^1$ gradient sums, over vertices, the largest absolute change to a neighbor within $A$, taking zero at an isolated vertex. The Dirichlet quotient divides this gradient by the $L^1$ norm of a nonzero function that vanishes at vertices adjacent to the complement. The Neumann quotient instead subtracts the average of a nonconstant function before taking its $L^1$ norm.

For each subgraph, take the infimum of the relevant quotients and multiply by $|A|$. At $n$, the Dirichlet profile takes the infimum of these values over $|A|\ge n$. The Neumann profile takes the supremum over $2\le|A|\le n$. We prove that the Dirichlet profile is comparable to $n/j(n)$ at the same $n$, and that the Neumann profile is at most a fixed multiple of it. This answers the $p=1$ case of Hume's Question~1.1 \cite{Hume2019}; the precise definitions and proofs are in \Cref{cor:separation-profile,cor:dirichlet-neumann}.

Coarea identifies $j(n)$, up to fixed factors, with the best $L^1$ norm-to-gradient ratio for nonzero functions supported on at most $n$ vertices, with the gradient taken on the whole graph. For every infinite finitely generated group, multiplying the argument of this normalized $L^1$ profile by a fixed $C>0$ therefore changes its value by at most a factor depending on $C$ (\Cref{norm:cor:lone}). This answers the first part of Correia's Question~1.3 and removes the regularity assumption in his Corollary~C \cite{Correia2025}.

The second part asks whether a related bound holds after composing one profile with the inverse of another. Here $f\simeq g$ means that $f(t)\le A g(Bt)$ and $g(t)\le A f(Bt)$ for some constants $A,B\ge1$ and all sufficiently large $t$. Suppose $j_G(v)/j_H(v)\to0$. Choose representatives $h_G\simeq j_G$ and $h_H\simeq j_H$, with $h_G$ nondecreasing and $h_H$ strictly increasing with an inverse defined at all sufficiently large values. Set $F=h_G\circ h_H^{-1}$: first find where $h_H$ reaches $t$, then evaluate $h_G$ there. Correia asks whether these representatives can always be chosen so that $F(Ct)=O_C(F(t))$ for every fixed $C>0$. For our infinite finitely generated elementary amenable examples, no such choice satisfies this bound (\Cref{prop:inverse-composition-counterexample}).

We prove the general profile estimates in \Cref{sec:atoms}; the example in \Cref{sec:homogeneity} shows why vertex transitivity is needed. The product results are in \Cref{sec:wreath}. In \Cref{sec:routing} we study orders, rays, nested exhaustions, and rearrangements. The group applications and the explicit polycyclic construction occupy \Cref{sec:stacking,sec:carpet,sec:fullgroups,sec:polycyclic}. The separation and Dirichlet bounds and the inverse-composition counterexample are in \Cref{sec:profile-consequences}. We close by comparing the right edge profile with the profile obtained by adding the left and right boundary counts in \Cref{sec:autocorrelation}.

\section{Preliminaries}\label{sec:conventions}

Let $\GammaG$ be an infinite group with a finite symmetric generating set $\cS$ such that $e\notin\cS$, and put $\Delta=\card{\cS}$. We use the right Cayley graph $x\longleftrightarrow xs$. Its word metric is left invariant, so left translations are graph automorphisms. Write $\card{g}_{\cS}=\dist(e,g)$ and $B(e,\rho)=\{g:\card{g}_{\cS}\le\rho\}$. We write $f\lesssim g$ if $f\le Cg$ for a constant $C>0$ independent of the varying parameter, and $f\gtrsim g$ for the reverse comparison.

For a finite set $Y\subset\GammaG$, write $\bd Y$ for the outer boundary from the Introduction and retain its right edge cut $\edgebd(Y)$. We also use the inner vertex boundary
\[
  \partial_{\cS}^{\mathrm{in}}Y=\{x\in Y:xs\notin Y\text{ for some }s\in\cS\}
\]
and the left edge cut
\[
  B^{\mathrm L}_{\cS}(Y)=\#\{(x,s)\in Y\times\cS:sx\notin Y\}.
\]

\begin{lemma}[Boundary identities]\label{lem:boundary-identities}
For every finite $Y\subset\GammaG$,
\begin{equation}\label{eq:right-symdiff}
  \edgebd(Y)=\frac12\sum_{s\in\cS}\card{Ys\mathbin\triangle Y},
\end{equation}
and
\begin{equation}\label{eq:vertex-edge}
\begin{split}
  \card{\bd Y}\le\edgebd(Y)\le\Delta\card{\bd Y},
  \qquad
  \card{\partial_{\cS}^{\mathrm{in}}Y}\le\edgebd(Y)\le\Delta\card{\partial_{\cS}^{\mathrm{in}}Y}.
\end{split}
\end{equation}
Moreover, for every $u\in\GammaG$,
\begin{equation}\label{eq:word-symdiff}
  \card{Yu\mathbin\triangle Y}\le2\card{u}_{\cS}\edgebd(Y).
\end{equation}
\end{lemma}

\begin{proof}
For each $s$, we have $\card{Ys\setminus Y}=\frac12\card{Ys\mathbin\triangle Y}$, which gives \eqref{eq:right-symdiff}. To see \eqref{eq:vertex-edge}, observe that every vertex of either vertex boundary is incident to between one and $\Delta$ crossing edges. Finally, writing $u=s_1\cdots s_\ell$ as a shortest word, we obtain
\[
  \card{Yu\mathbin\triangle Y}
  \le\sum_{i=1}^{\ell}\card{Ys_i\cdots s_\ell\mathbin\triangle Ys_{i+1}\cdots s_\ell}
  =\sum_{i=1}^{\ell}\card{Ys_i\mathbin\triangle Y}
  \le2\ell\edgebd(Y).
\]
Here the equality follows by right cancellation.
\end{proof}

For $n\ge1$, we use the profiles $\Iex$, $\Iup$, $\IexE$ and $\IupE$ from the Introduction. They take positive integer values, and the infima defining $\Iup$ and $\IupE$ are attained. By \eqref{eq:vertex-edge},
\begin{equation}\label{eq:profile-comparison}
\Iex(n)\le\IexE(n)\le\Delta\Iex(n),\qquad
\Iup(n)\le\IupE(n)\le\Delta\Iup(n).
\end{equation}

The exact profile need not be nondecreasing. For $\Z\times(\Z/2\Z)$ with generators $(\{-1,0,1\}\times\Z/2\Z)\setminus\{(0,0)\}$, $I(2k)=4$ and $I(2k-1)=5$. Indeed, if $J$ is the projection of $F$ onto $\Z$, its closed neighborhood is $(J+\{-1,0,1\})\times\Z/2\Z$, whence $|\partial F|\ge2\lceil|F|/2\rceil+4-|F|$; consecutive full columns and at most one partial column attain equality.

Subscripts or superscripts on profiles specify the underlying group or graph and, when needed, its generating set; $\Delta_G$ denotes the degree of the chosen Cayley graph of $G$.

\begin{proposition}[Bilipschitz comparison]\label{prop:bilipschitz-comparison}
Let $(\Gamma,\cS)$ and $(\Lambda,\cT)$ be infinite finitely generated groups with their word metrics, and let $\phi\colon\Gamma\to\Lambda$ be an $L$-bilipschitz bijection. There is $C<\infty$ such that, for every $n\ge1$,
\[
 C^{-1}I^{\Gamma,\cS}(n)\le I^{\Lambda,\cT}(n)\le C I^{\Gamma,\cS}(n).
\]
The same inequalities hold with $I$ replaced by $I_{\uparrow}$. Thus boundedness of $I/I_{\uparrow}$ is a bilipschitz invariant and, for a fixed group, does not depend on the finite generating set.
\end{proposition}

\begin{proof}
Put $R=\lceil L\rceil$. If $F\subset\Gamma$ is finite and nonempty and $y\in\partial_{\cT}\phi(F)$, then $\phi^{-1}(y)\notin F$ is within distance $R$ of $F$. A shortest path to $F$ meets $\partial_{\cS}F$, so
\[
 \card{\partial_{\cT}\phi(F)}\le\card{B_{\cS}^{\Gamma}(e,R)}\card{\partial_{\cS}F}.
\]
We now apply the same argument to $\phi^{-1}$, minimize over $|F|=n$, and then take infima over $m\ge n$. This gives both profile inequalities. The last assertion follows because the identity between two word metrics on a finitely generated group is bilipschitz.
\end{proof}

\begin{lemma}[Adding and deleting vertices]\label{lem:boundary-editing}
Adding one vertex increases the outer vertex boundary by at most $\Delta$. Adding or deleting one vertex changes the edge cut by at most $\Delta$. Consequently,
\[
  \Iex(n+1)\le\Iex(n)+\Delta,
  \qquad
  \IexE(n)\le\IexE(m)+\Delta(m-n)\quad(m\ge n).
\]
\end{lemma}

\begin{proof}
If $x\notin Y$, then $\bd(Y\cup\{x\})\subseteq(\bd Y\setminus\{x\})\cup x\cS$. Adding or deleting a vertex affects only its $\Delta$ incident edges. The inequalities follow by iteration.
\end{proof}

\begin{lemma}[Subadditivity and amenability {\cite{Gromov2008}}]\label{lem:subadditivity-amenability}
For $m,n\ge1$, $\Iex(m+n)\le\Iex(m)+\Iex(n)$, and
\[
  \lim_{n\to\infty}\frac{\Iex(n)}n=\inf_{n\ge1}\frac{\Iex(n)}n.
\]
This limit is zero if and only if $\GammaG$ is amenable.
\end{lemma}

\begin{proof}
Taking disjoint left translates of minimizing sets gives subadditivity. Fekete's lemma then gives the asserted limit, and the F{\o}lner criterion identifies when it is zero.
\end{proof}

If $h:=\inf_{n\ge1}\Iex(n)/n>0$, then $\Iup(n)\ge hn$ and $\Iex(n)\le\Delta n$. Thus nonamenable groups satisfy
\begin{equation}\label{eq:nonamenable-ratio}
  \sup_{n\ge1}\frac{\Iex(n)}{\Iup(n)}\le\frac{\Delta}{h}.
\end{equation}

\begin{proposition}[Growth lower bounds]\label{prop:exp-profile-lower}
Put $\rho(v)=\min\{r\in\N:|B(e,r)|>v\}$. For every $n\ge1$,
\begin{equation}\label{eq:csc-edge}
 I_E(n)\ge\frac{n}{2\rho(2n)},\qquad
 I_{E,\uparrow}(n)\ge\frac{n}{12\rho(2n)}.
\end{equation}
Consequently $I(n),I_\uparrow(n)\ge n/(12\Delta\rho(2n))$. In exponential growth, for all sufficiently large $n$,
\begin{equation}\label{eq:exp-profile-lower}
 I(n),I_\uparrow(n)\ge c\frac{n}{\log(\e+n)}.
\end{equation}
\end{proposition}
\begin{proof}
We use the argument of Coulhon--Saloff-Coste \cite{CoulhonSaloffCoste1993}. For $|F|=m$ and $r=\rho(2m)$, counting pairs gives $\sum_{u\in B(e,r)}|F\cap Fu|\le m^2$. Since $|B(e,r)|>2m$, some $u$ has $|F\cap Fu|<m/2$, hence $|F\triangle Fu|>m$. By \eqref{eq:word-symdiff}, $|F\triangle Fu|\le2rB_{\cS}(F)$.

To pass to the monotone profile, we need to compare the inverse growth function at different cardinalities. For $m\ge n$, put $r=\rho(2n)$ and $M=\lceil m/n\rceil$, and choose $M$ vertices at distances $i(2r+1)$, $0\le i<M$, on a geodesic ray from $e$. Their radius-$r$ balls are disjoint and each has more than $2n$ points, while all of them lie in $B(e,M(2r+1))$. We therefore have
\[
 \rho(2m)\le M(2r+1)\le6(m/n)\rho(2n).
\]
The second inequality in \eqref{eq:csc-edge} now follows from the first by taking the infimum over $m\ge n$. The vertex estimates follow from \eqref{eq:profile-comparison}. Exponential growth gives $\rho(2n)=O(\log(\e+n))$.
\end{proof}

\begin{lemma}[Bounded monotone profiles]\label{lem:bounded-profile-nesting}
Let $\GammaG$ be an infinite finitely generated group with a finite symmetric generating set $\cS$. If $\IupE(n)$ is bounded, there are $C=C(\GammaG,\cS)<\infty$ and an enumeration $\GammaG=\{g_1,g_2,\ldots\}$ such that
\[
 B_{\cS}(\{g_1,\ldots,g_n\})\le C\qquad(n\ge1).
\]
\end{lemma}

\begin{proof}
Suppose $a(n)=\IupE(n)\le M$. For an integer $r\ge0$, put $n=\lceil12M(r+1)\rceil$. By \Cref{prop:exp-profile-lower}, $\rho(2n)\ge r+1$, so $|B(e,r)|\le2n=O(r+1)$. Thus $\GammaG$ has linear growth and contains an infinite cyclic subgroup $H=\langle t\rangle$ of finite index \cite[Theorem~1.2]{TointonYadin2016}.

Choose representatives $Q$ with $\GammaG=\bigsqcup_{q\in Q}Hq$. For $q\in Q$ and $s\in\cS$, write $qs=t^{\ell(q,s)}q'$ with $q'\in Q$. Then $(t^kq)s=t^{k+\ell(q,s)}q'$. Put
\[
 E_R=\{t^kq:|k|\le R,\ q\in Q\}\qquad(R=0,1,\ldots).
\]
For each pair $(q,s)$, at most $|\ell(q,s)|$ of the exponents $k\in[-R,R]$ are sent outside this interval. Hence
\[
 B_{\cS}(E_R)\le D:=\sum_{q\in Q,\,s\in\cS}|\ell(q,s)|,
 \qquad |E_{R+1}\setminus E_R|=2|Q|.
\]
Order $E_0$ arbitrarily and then insert each successive difference one point at a time. Initial segments inside $E_0$ have cut at most $\Delta|Q|$, and all later ones have cut at most $D+2\Delta|Q|$ by \Cref{lem:boundary-editing}. Since the sets $E_R$ exhaust $\GammaG$, this gives the required enumeration.
\end{proof}

\section{Edge atoms and prescribed cardinality}\label{sec:atoms}
We now work with an infinite connected vertex-transitive graph $X$ of finite degree $\Delta\ge2$. The bounded-degree amenable graph in \Cref{sec:homogeneity} will show why we need the transitivity assumption. Write $B(F)=|E_X(F,X\setminus F)|$ and use the profiles $I_E$, $a$ and $j$ defined in the Introduction, with this edge boundary. Their extrema are attained, since $B(F)$ is a positive integer and there are only finitely many possible pairs $(|F|,B(F))$ with $|F|\le n$. Set $a(0)=0$. Given two finite sets, an automorphism moves the first off the second: it suffices to move one of its vertices to distance greater than its diameter from the second set. An \emph{edge $n$-atom} is a set of least cardinality attaining $a(n)$.

In the following lemma, we apply the edge-atom argument of \cite[Lemma~1 and p.~142]{HLST2000} to finite sets in an infinite graph. We compare an atom with its automorphic images: a proper intersection gives a smaller set with controlled boundary, while the absence of such intersections forces a lower bound on the boundary of the atom.
\begin{lemma}[Edge atoms]\label{norm:lem:atom}
Let $A$ be an edge $n$-atom, $M=|A|$ and $b=a(n)$. At least one of the following holds:
\begin{enumerate}[(i)]
\item $M=n$;
\item the automorphic images of $A$ are disjoint or equal, and $b\ge M$;
\item some $Q$ satisfies $n/2<|Q|<n$ and $B(Q)\le b$.
\end{enumerate}
\end{lemma}
\begin{proof}
The edge boundary satisfies the submodular and posimodular inequalities:
\begin{align*}
 B(U\cap V)+B(U\cup V)&\le B(U)+B(V),\\
 B(U\setminus V)+B(V\setminus U)&\le B(U)+B(V).
\end{align*}
Suppose first that an automorphic image $A'$ meets $A$ properly. By submodularity and $|A\cup A'|\ge n$, we have $B(A\cap A')\le b$, so minimality of $M$ gives $t:=|A\cap A'|<n$. The second inequality gives $B(E)\le b$ for one of the two differences $E$, and the same minimality argument gives $|E|=M-t<n$. It follows that $M\le2n-2$; in particular, images are disjoint or equal when $M\ge2n-1$. Now, if $M>n$, the larger of $E$ and $A\cap A'$ has cardinality at least $M/2>n/2$, which proves (iii).

The case $M=n$ is precisely alternative (i). Thus it remains to consider $M>n$ when there is no proper nonempty intersection. In this case $A$ is a finite block, and its setwise stabilizer is transitive on $A$: an automorphism taking one point of $A$ to another must preserve $A$. Every vertex of $A$ therefore has the same external degree. This degree is at least one by connectivity and infinitude, which gives $b\ge M$.
\end{proof}

\subsection{Normalized profiles}
\begin{theorem}\label{norm:thm:universal}
For all integers $1\le u\le v$,
\begin{align}
 j(v)&\le\Delta\frac vu j(u),\label{norm:eq:linear}\\
 \frac{n}{\Delta j(n)}&\le a(n)\le\frac{2n}{j(n)}\qquad(n\ge1).
 \label{norm:eq:minorant}
\end{align}
In particular, $j(2n)\le2\Delta j(n)$ and $j(Cv)=O_C(j(v))$ for every fixed $C>0$.
\end{theorem}

\begin{proof}
We begin with the consequence of \Cref{norm:lem:atom} that $j(n)\ge n/(\Delta a(n))$. In alternatives (i) and (iii), we use $A$ or $Q$, together with $\Delta\ge2$. In alternative (ii), we use $a(n)\ge n$ and $j(n)\ge1/\Delta$. This proves the lower bound in \eqref{norm:eq:minorant}.

To compare the normalized profiles at two cardinalities, choose $F$ attaining $j(v)$, and put $m=\card F$ and $c=\cut(F)$. If $m\le u$, then $j(v)=j(u)$ and there is nothing to prove. Otherwise we have $a(u)\le c$, and hence
\[
 j(v)=m/c\le v/a(u)\le\Delta(v/u)j(u).
\]
This proves \eqref{norm:eq:linear}.

For the upper bound in \eqref{norm:eq:minorant}, let $E$ attain $j(n)$ and write $k=\card E\le n$. Take $\lceil n/k\rceil$ disjoint automorphic images of $E$. Their union has size at least $n$ and cut at most
\[
 \lceil n/k\rceil\cut(E)\le2n\frac{\cut(E)}k=2n/j(n).
\]
The remaining assertions follow from \eqref{norm:eq:linear} and monotonicity, taking integer parts for real arguments.
\end{proof}

\begin{corollary}[Normalized $L^1$ group profiles]\label{norm:cor:lone}
Let $G$ be any infinite finitely generated group with a finite symmetric set $S$ of generators, $e\notin S$, and $\Delta=\card S$. Its normalized outer-boundary profile
\[
 J_G(n)=\max_{0<\card F\le n}\frac{\card F}{\card{\partial_S F}}
\]
satisfies $J_G(2n)\le2\Delta^2J_G(n)$ and $J_G(Cv)=O_C(J_G(v))$ for every fixed $C>0$.
\end{corollary}
\begin{proof}
For the right Cayley graph, $\card{\partial_S F}\le\cut(F)\le\Delta\card{\partial_S F}$, so $j_G\le J_G\le\Delta j_G$.
\end{proof}

The discrete coarea formula gives
\[
 \sum_{x\in G,s\in S}\card{f(xs)-f(x)}
 =2\int_0^\infty\cut(\{f>t\})\,dt
\]
for nonnegative finitely supported $f$. Thus $\sup\{\|f\|_1/\sum_{x,s}|f(xs)-f(x)|:0<|\supp f|\le n\}=j_G(n)/2$. Replacing the sum over generators by a maximum changes this profile by at most a factor depending on $\Delta$.

The profile in \cite{Correia2025} uses $|SF\triangle F|$. It is comparable to $J_G$, since inversion exchanges left and right multiplication and
\[
 |SF\setminus F|\le |SF\triangle F|\le2|SF\setminus F|.
\]
Fixed argument rescalings are absorbed by \eqref{norm:eq:linear}, so the regularity is unchanged under the asymptotic equivalence used there.

\subsection{Sets of prescribed cardinality}
\begin{lemma}\label{lem:step}
For each $n\ge1$, either $I_E(n)\le\Delta a(n)$ or some integer $1\le r<n/2$ satisfies
\[
 I_E(n)\le a(n)+I_E(r).
\]
\end{lemma}
\begin{proof}
The first two alternatives of \Cref{norm:lem:atom} give $I_E(n)\le\Delta a(n)$. In the third alternative, we still have points to add. Put $r=n-|Q|$, choose an $r$-point set with edge cut $I_E(r)$, and take an automorphic image of it disjoint from $Q$. The union has the required cardinality, and its edge cut is at most the sum of the two edge cuts. This proves the inequality.
\end{proof}

\begin{theorem}[Dyadic estimate]\leavevmode\phantomsection\label{thm:dyadic}\ignorespaces
For every $n\ge1$,
\begin{equation}\label{eq:dyadic}
 \IexE(n)\le\Delta\sum_{j\ge0}a\!\left(\left\lfloor\frac n{2^j}\right\rfloor\right).
\end{equation}
In particular,
\begin{equation}\label{eq:log-density}
 \frac{\IexE(n)}{a(n)}
 \le\Delta+\left\lceil\log_2\!\left(\frac{\Delta n}{a(n)}\right)\right\rceil.
\end{equation}
\end{theorem}
\begin{proof}
We apply \Cref{lem:step} repeatedly. Each application of its second alternative more than halves the positive integer cardinality, so after finitely many steps the first alternative must hold. The cardinalities $n_i$ satisfy $n_0=n$ and $1\le n_{i+1}<n_i/2$. Adding the resulting bounds gives $I_E(n)\le\sum_{i<\ell}a(n_i)+\Delta a(n_\ell)$. Now $n_i\le\lfloor n/2^i\rfloor$, so monotonicity gives \eqref{eq:dyadic}, with a finite sum because $a(0)=0$.

For \eqref{eq:log-density}, we use the same argument but stop once the remaining cardinality is sufficiently small. Put $b=a(n)$ and $J=\lceil\log_2(\Delta n/b)\rceil$. We stop when the first alternative of \Cref{lem:step} holds, or after $J$ applications of the second alternative. In the first case the last term is at most $\Delta b$. In the second, the remaining cardinality is at most $n/2^J\le b/\Delta$, so its exact edge profile is at most $b$. Since each preceding term is also at most $b$, we obtain the stated bound.
\end{proof}

A positive function $f$ is \emph{regularly varying of index $\alpha$} if $f(\lambda t)/f(t)\to\lambda^\alpha$ as $t\to\infty$ for every fixed $\lambda>0$.

\begin{corollary}[A sufficient condition for comparison]\label{cor:scaling}
The comparison $\IexE(n)=O(a(n))$ holds if
\begin{equation}\label{eq:dini}
 \sum_{j\ge0}a\!\left(\left\lfloor\frac n{2^j}\right\rfloor\right)=O(a(n)).
\end{equation}
Condition~\eqref{eq:dini} holds if there are $q>1$, $0<\theta<1$ and $n_0$ such that
\begin{equation}\label{eq:scaling}
 a(\lfloor n/q\rfloor)\le\theta a(n)\qquad(n\ge n_0).
\end{equation}
In particular, comparison holds if $a(n)\asymp f(n)$ for a regularly varying function $f$ of positive index.
\end{corollary}
\begin{proof}
The first assertion follows from \Cref{thm:dyadic}. To prove the second, choose a positive integer $k$ with $p=2^k\ge q$. By monotonicity, $a(\lfloor n/p\rfloor)\le\theta a(n)$ above $n_0$. Since $p$ is an integer, we can iterate this to obtain $a(\lfloor n/p^i\rfloor)\le\theta^i a(n)$ as long as the arguments are at least $n_0$. We now divide the dyadic sum into blocks of $k$ terms and bound each block by $k$ times its first term. The part left over is bounded by a fixed constant, hence by a fixed multiple of $a(n)$ because $a(n)\ge1$.

If $a\asymp f$ and $f$ is regularly varying with positive index $\alpha$, choose $q$ large enough that the comparison constants times $q^{-\alpha}$ are less than one; then \eqref{eq:scaling} follows for large $n$.
\end{proof}

\noindent\emph{Remark.} Conversely, if the sum in \eqref{eq:dini} is at most $Ca(n)$ for large $n$, choose an integer $k>2C$. Its first $k$ terms are each at least $a(\lfloor n/2^k\rfloor)$, so $a(\lfloor n/2^k\rfloor)\le a(n)/2$. Thus conditions \eqref{eq:dini} and \eqref{eq:scaling} are equivalent. This sufficient condition is not necessary for comparison: on $\Z$ with $S=\{\pm1\}$, $I_E=a=2$, whereas the sum in \eqref{eq:dini} is unbounded.

For Cayley graphs, condition~\eqref{eq:dini} will also give finite-set orderings and nested exhaustions by \Cref{thm:dini-orderings,cor:dini-nesting}.

\subsection{Large values of the comparison ratio}
\begin{theorem}[A lower bound for the monotone profile]\label{thm:records}
Suppose
\[
 R=\frac{\IexE(n)}{a(n)}
 =\max_{1\le r\le n}\frac{\IexE(r)}{a(r)},
 \qquad R\ge2\Delta+2.
\]
Set $J=\lfloor R/2\rfloor$. Then
\begin{equation}\label{eq:plateau}
 a\!\left(\left\lfloor\frac n{2^J}\right\rfloor\right)\ge\tfrac12a(n).
\end{equation}
If $\IexE/a$ is unbounded, one can choose such $n$ with $R\to\infty$. For each of them, $a(n)/2\le a(m)\le a(n)$ whenever $\lfloor n/2^J\rfloor\le m\le n$, and $n/\lfloor n/2^J\rfloor\ge2^J$.
\end{theorem}
\begin{proof}
We again apply \Cref{lem:step} repeatedly, starting from $n$. This time we use the large comparison ratio to show that the first alternative cannot occur too soon. After $j$ applications of the second alternative,
\[
 \IexE(n_j)\ge\IexE(n)-\sum_{i<j}a(n_i)
 \ge(R-j)a(n),\qquad n_j\le\lfloor n/2^j\rfloor.
\]
If the first alternative held at a step $j\le J$, it would give $\IexE(n_j)\le\Delta a(n_j)\le\Delta a(n)$, contrary to $R-j\ge R/2>\Delta$. The recursion therefore reaches $n_J\ge1$. The choice of $n$ gives $\IexE(n_J)\le R a(n_J)$, so
\[
 a(n_J)\ge\left(1-\frac JR\right)a(n)\ge\tfrac12a(n).
\]
Monotonicity proves \eqref{eq:plateau} and the bounds throughout the stated interval. Since $\lfloor n/2^J\rfloor\ge n_J\ge1$, the interval spans a multiplicative factor of at least $2^J$.
\end{proof}

\begin{corollary}[Bounds from growth]\label{cor:growth}
For a Cayley graph, let $\rho$ be the inverse growth function from \Cref{prop:exp-profile-lower}. Then for every $n\ge1$,
\[
 \frac{\IexE(n)}{a(n)}
 \le\Delta+\left\lceil\log_2\!\left(12\Delta\rho(2n)\right)\right\rceil.
\]
If $G$ has exponential growth, then
\[
 \frac{\IexE(n)}{a(n)}=O\!\left(\log\log(e^e+n)\right)
\]
for all $n$.
\end{corollary}
\begin{proof}
Combining \Cref{prop:exp-profile-lower} with \eqref{eq:log-density} gives the first bound. For a group of exponential growth, we have $\rho(2n)\le C\log(e+n)$, which gives the second assertion.
\end{proof}

\Cref{thm:main-universal} follows from \Cref{norm:thm:universal,thm:dyadic,cor:scaling}.

\subsection{A nontransitive example}\label{sec:homogeneity}

Transitivity cannot be omitted from the general profile estimates. The comparison fails for amenable graphs of bounded degree. For an infinite locally finite graph $X$ and a finite set $Y\subset V(X)$, let $B_X(Y)$ be the number of edges with exactly one endpoint in $Y$. Write $I_E^X$ and $I_{E,\uparrow}^X$ for its exact and monotone edge profiles, using the same definitions as above. Here $X$ is called amenable if $\inf_{\varnothing\ne Y\subset_{\mathrm{fin}}V(X)}B_X(Y)/\card Y=0$.

Let $H_n=(V_n,E_n)$ be finite connected graphs of uniformly bounded degree, with $N_n=\card{V_n}$. Write $\partial^{\mathrm e}_{H_n}A$ for the edges of $H_n$ with exactly one endpoint in $A$. Assume a uniform edge-expansion constant $h>0$:
\begin{equation}\label{eq:graph-expansion}
  \card{\partial^{\mathrm e}_{H_n}A}
  \ge h\min\{\card A,N_n-\card A\}
  \qquad(A\subset V_n).
\end{equation}
Bounded-degree expander families of unbounded order exist \cite{HooryLinialWigderson2006}. Passing to a subsequence, assume that
\begin{equation}\label{eq:graph-superlacunary}
  N_n\ge8\sum_{j<n}N_j.
\end{equation}
Join a distinguished vertex of $H_n$ to one of $H_{n+1}$ by a bridge edge, to obtain an infinite graph $X$. Put
\[
  S_n:=\sum_{j\le n}N_j,
  \qquad
  r_n:=\lfloor N_n/2\rfloor.
\]

\begin{proposition}
\label{prop:bounded-degree-graph}
The graph $X$ is amenable and, for all sufficiently large $n$,
\[
  I_{E,\uparrow}^X(r_n)=1,
  \qquad
  I_E^X(r_n)\ge\frac{hN_n}{5}.
\]
Consequently, $I_E^X/I_{E,\uparrow}^X$ is unbounded.
\end{proposition}

\begin{proof}
We first observe that the union $V_1\cup\cdots\cup V_n$ has only one edge leaving it. This shows that $X$ is amenable and that $I_{E,\uparrow}^X(r_n)\le I_E^X(S_n)\le1$. On the other hand, since $X$ is connected and infinite, every finite nonempty set has a crossing edge. Thus $I_{E,\uparrow}^X(r_n)=1$.

Let $Y\subset V(X)$ have $\card Y=r_n$. Since $N_j\ge N_n$ for $j\ge n$, we have $\card{Y\cap V_j}\le r_n\le N_j/2$. Applying \eqref{eq:graph-expansion} in these blocks gives
\[
 B_X(Y)\ge h\sum_{j\ge n}\card{Y\cap V_j}
 \ge h(r_n-S_{n-1})
 \ge h(3N_n/8-1).
\]
This is at least $hN_n/5$ for all sufficiently large $n$. Taking the minimum over $Y$ proves the lower bound.
\end{proof}

Here $I_{E,\uparrow}^X(n)=1$ for every $n$, so the dyadic estimate of \Cref{thm:dyadic} also fails.

\section{Direct and wreath products}\label{sec:wreath}
\subsection{Direct products}

The two-scale estimate for $j$ has a useful consequence for direct products. We use the classical product-profile formula of Coulhon--Grigor'yan--Levin \cite[Theorem~1.1 and \S7]{CoulhonGrigoryanLevin2003}.

\begin{proposition}[Direct products]\label{prop:mixed-products}
Let $G=G_1\times\cdots\times G_k$, where $k\ge2$ and each $G_i$ is infinite and finitely generated. For every finite symmetric generating set of $G$, there are $C,c>0$ such that, for all sufficiently large $v\ge u$,
\[
 j_G(v)\le C\left(\frac vu\right)^{1/k}j_G(u),\qquad
 a_G(v)\ge c\left(\frac vu\right)^{1-1/k}a_G(u).
\]
Consequently $I_E^G(n)\asymp a_G(n)$ and $I^G(n)\asymp I_\uparrow^G(n)$.
\end{proposition}
\begin{proof}
We use a product generating set and write $\Delta_i$ for the degree of the chosen Cayley graph of $G_i$; changing generators changes the profiles only by fixed multiplicative constants. If one factor is nonamenable, so is $G$, and the assertions follow from boundedness of $j_G$ and $a_G(n)\asymp n$. Suppose therefore that every factor is amenable. We first replace the factor profiles by comparable concave functions. Their inverses describe the volumes needed to reach a common value $t$. Each inverse grows at least linearly, so raising $t$ by a factor $\lambda\ge1$ requires a factor at least $\lambda^k$ in the product volume. For $x\ge1$, set
\[
 h_i(x)=\inf_{n\in\N,\ n\ge1}\left(1+\frac xn\right)j_{G_i}(n).
\]
Each $h_i$ is nondecreasing and concave. For integral $x$, \eqref{norm:eq:linear} and monotonicity give
\[
 \Delta_i^{-1}j_{G_i}(x)\le h_i(x)\le2j_{G_i}(x).
\]
Indeed, when $n\le x$ the term $(x/n)j_{G_i}(n)$ gives the lower bound, and when $n\ge x$ the term $j_{G_i}(n)$ does. Integer parts extend the comparison to real $x$. Extend $h_i$ linearly from $(0,0)$ to $(1,h_i(1))$. This preserves concavity, since $h_i(x)/x$ is nonincreasing for $x\ge1$. Amenability makes $h_i$ unbounded, so concavity makes it strictly increasing.

Its inverse $q_i=h_i^{-1}$ is convex and vanishes at zero. Put
\[
 Q(t)=\prod_{i=1}^k q_i(t),\qquad H=Q^{-1}.
\]
Since $q_i(t)/t$ is nondecreasing, $Q(t)/t^k$ is nondecreasing. Thus
\[
 H(v)\le(v/u)^{1/k}H(u)\qquad(v\ge u>0).
\]
The product-profile formula, with counting measure and the sum gradient, says that $1/j_G(v)$ is comparable, allowing a fixed rescaling of $v$, to
\[
 \inf_{x_1\cdots x_k=v}\sum_{i=1}^k\frac1{h_i(x_i)}.
\]
The sum is comparable to its largest term. The infimum of that largest term is exactly $1/H(v)$: the choice $x_i=q_i(H(v))$ attains it, while any larger common lower bound for the $h_i(x_i)$ would give $x_1\cdots x_k>v$. For large $v$, all the chosen $x_i$ are at least two. Replacing them by integer parts changes the factor profiles by at most fixed constants, by concavity; products of sets attaining these factor profiles give the corresponding upper estimate for the boundary-to-volume ratio. Thus the extensions near zero do not affect the comparison. The displayed scaling inequality for $H$ absorbs the fixed argument rescaling in the product theorem. Hence $j_G\asymp H$, which proves the first assertion.

The comparison $a_G(n)\asymp n/j_G(n)$ from \eqref{norm:eq:minorant} gives the second assertion. Choose a fixed $q$ large enough that $a_G(qn)\ge2a_G(n)$ for large $n$. Then \Cref{cor:scaling} gives the exact-profile conclusions.
\end{proof}

\subsection{Wreath products}
We first fix the convention for the F{\o}lner function. For a finitely generated group $G$, we use outer vertex boundary and a real parameter $t$, and put
\[
 \Phi_G(t)=\min\{\card F:0<\card F<\infty,
                   \ \card{\partial F}/\card F\le1/t\}\quad(t>0).
\]
We set $\Phi_G(t)=+\infty$ if no such set exists. If $G$ is infinite and amenable, $\Phi_G$ is finite and tends to infinity. For infinite $G$, we have
\begin{equation}\label{norm:eq:inverse}
 \Phi_G(t)\le v\ \Longleftrightarrow\ J_G(v)\ge t
 \qquad(v\in\N).
\end{equation}
For a finite group $\Phi_G$ eventually equals its order. We use $\simeq$ in the sense fixed in the Introduction.

\begin{theorem}[Wreath products]\label{norm:thm:wreath}
Let $\Gamma$ be a nontrivial finitely generated group and $H$ an infinite finitely generated group. The restricted wreath product $W=\Gamma\wr H$ satisfies, for every finite symmetric generating set,
\[
 \IexE^W(n)=O(a_W(n)),\qquad I^W(n)=O(I_\uparrow^W(n)).
\]
For finite nontrivial $\Gamma$ and amenable $H$,
\begin{equation}\label{norm:eq:finite-lamps}
 a_W(n)\asymp\frac{n}{j_H(\lfloor\log n\rfloor)}.
\end{equation}
\end{theorem}
\begin{proof}
We may use the standard move and switch generators: a move changes the $H$-coordinate by a generator of $H$, and a switch changes the $\Gamma$-value of the function at that coordinate by a generator of $\Gamma$. Equation~\eqref{eq:word-symdiff} shows that edge boundaries for two finite generating sets are comparable. If either factor is nonamenable, then so is $W$, because it contains a copy of $\Gamma$ and maps onto $H$. In that case there is $h>0$ with $B(F)\ge h|F|$ for every finite nonempty $F\subset W$, so $a_W(n)\ge hn$, while $\IexE^W(n)\le\Delta_W n$, and the comparison follows. We may therefore assume for the rest of the proof that both factors are amenable.

Erschler's estimate \cite[Theorem~1]{Erschler2003}, as formulated in \cite[Theorem~2.19]{CorreiaDumoncel2026}, states
\begin{equation}\label{norm:eq:erschler}
 \Phi_W(t)\simeq P(t),\qquad
 P(t)=\Phi_\Gamma(t)^{\Phi_H(t)},
\end{equation}
provided $J_H(Cv)=O_C(J_H(v))$ for every fixed $C>0$. This hypothesis holds by \Cref{norm:cor:lone}.

The next estimate will bound $j_W(qn)/j_W(n)$ by a constant independent of the fixed integer $q>1$, since $qn\le n^2$ once $n\ge q$. As $a_W(n)\asymp n/j_W(n)$, choosing $q$ large will then give $a_W(qn)\ge2a_W(n)$ for all sufficiently large $n$. We therefore prove that, for some $K_W<\infty$,
\begin{equation}\label{norm:eq:power}
 j_W(n^2)\le K_Wj_W(n)
\end{equation}
for all sufficiently large integers $n$.

It is convenient to do this first for the F{\o}lner function: we will show that $\Phi_W(Ct)\ge\Phi_W(t)^2$ for some $C\ge1$ and all sufficiently large $t$, and then use \eqref{norm:eq:inverse} to obtain \eqref{norm:eq:power}. Choose $K\ge1$ with $J_H(2r)\le KJ_H(r)$ for integers $r\ge1$. For large $t$, let $r=\Phi_H(t)\ge2$. Equation~\eqref{norm:eq:inverse} gives $J_H(r-1)<t$, and since $r\le2(r-1)$, we have
\[
 J_H(r)<Kt,\qquad J_H(4r)<K^3t.
\]
Therefore, with $L=K^3$,
\begin{equation}\label{norm:eq:base-growth}
 \Phi_H(Lt)>4\Phi_H(t).
\end{equation}
As $\Phi_\Gamma$ is nondecreasing and eventually at least two, \eqref{norm:eq:base-growth} implies $P(Lt)\ge P(t)^4$.

Choose $A,B\ge1$ so that \eqref{norm:eq:erschler} gives
\[
 \Phi_W(t)\le A P(Bt),\qquad P(t)\le A\Phi_W(Bt)
\]
for large $t$. Put $C=LB^2$. Then
\begin{align*}
 \Phi_W(Ct)&\ge A^{-1}P(LBt)
 \ge A^{-1}P(Bt)^4
 \ge A^{-5}\Phi_W(t)^4
 \ge\Phi_W(t)^2
\end{align*}
for large $t$, the last inequality using $\Phi_W(t)\to\infty$. Fix a large integer $v$ and take $t>J_W(v)$. Equation~\eqref{norm:eq:inverse} gives $\Phi_W(t)>v$, hence $\Phi_W(Ct)>v^2$ and $J_W(v^2)<Ct$. Letting $t\downarrow J_W(v)$ yields $J_W(v^2)\le C J_W(v)$. Comparison of edge and vertex boundaries proves \eqref{norm:eq:power}, for instance with $K_W=C\Delta_W$.

We next show that $a_W(qn)\ge2a_W(n)$ for a fixed integer $q>1$ and all sufficiently large $n$. Choose an integer $q\ge4\Delta_WK_W$. For sufficiently large $n\ge q$, we have $qn\le n^2$, so \Cref{norm:thm:universal} and \eqref{norm:eq:power} give
\[
 a_W(qn)\ge\frac{qn}{\Delta_W j_W(qn)}
 \ge\frac{qn}{\Delta_WK_W j_W(n)}
 \ge\frac{q}{2\Delta_WK_W}a_W(n)\ge2a_W(n).
\]
For $m$ large, set $n=\lfloor m/q\rfloor$. We then have $a_W(n)\le a_W(qn)/2\le a_W(m)/2$, so \Cref{cor:scaling} gives the exact edge comparison. Comparing vertex and edge boundaries gives the assertion for the outer boundary as well.

For finite $\Gamma$, inverting \eqref{norm:eq:erschler} gives $J_W(v)\simeq J_H(\log v)$ \cite[p.~20]{CorreiaDumoncel2026}. The rescalings of the argument can be omitted, since $\lfloor\log(qn)\rfloor\le2\lfloor\log n\rfloor$ eventually for each fixed $q>1$, and $J_H(Cv)=O_C(J_H(v))$ by \Cref{norm:cor:lone}. Thus $j_W(n)\asymp j_H(\lfloor\log n\rfloor)$. Equation~\eqref{norm:eq:minorant} proves \eqref{norm:eq:finite-lamps}.
\end{proof}

\begin{remark}\label{rmk:wreath-power-consequences}
For the infinite-base wreath products in \Cref{norm:thm:wreath}, \eqref{norm:eq:power} and monotonicity give
\[
 j_W(v^\alpha)\asymp j_W(v)\quad(\alpha>0),\qquad
 j_W(v/\log v)\asymp j_W(v).
\]
For the second comparison, use $\sqrt v\le v/\log v\le v$ for large $v$. These give affirmative answers at $p=1$ to Questions~5.8 and~5.15 of \cite{CorreiaDumoncelProfiles2026} for this class of amenable groups.

Let $W=\Gamma\wr H$, where $\Gamma$ is a nontrivial finitely generated amenable group and $H$ is an infinite finitely generated amenable group. The constructions below are those of \cite{CorreiaDumoncelProfiles2026}. A lampjuggler uses the group of finitely supported permutations of $W\times\{1,\ldots,s\}$, with an integer $s\ge1$. A designer uses the semidirect product of finitely supported functions $W\to F$, for a finite nontrivial group $F$, with finitely supported permutations of $W$ acting by permuting their arguments. A cloner uses the group of linear automorphisms of a vector space over a finite field with basis $W$ that fix all but finitely many basis vectors. The construction takes the semidirect product of each of these groups with $W$, acting by shifting indices. Theorems~5.6, 5.11 and~5.13 of that paper now apply to these constructions. If $K$ is obtained by a fixed sequence of $r\ge1$ such constructions, they give
\[
 j_K(v)\asymp j_W(\log^{\circ r}v),\qquad
 I_E^K(n)\asymp a_K(n)\asymp\frac n{j_W(\log^{\circ r}n)},
\]
where $\log^{\circ r}$ denotes the $r$-fold iterate. Indeed, at each stage the fixed-dilation hypothesis follows from \Cref{norm:thm:universal}. Each construction preserves infinitude, finite generation and amenability. Its one-step profile is comparable to the base profile composed with $\log$, so the fixed-dilation estimate for the base profile gives the square estimate for the resulting profile. The comparisons with $\sqrt v$ and $v/\log v$ then follow by monotonicity, proving the claim by induction on the fixed sequence. The calculation in the proof of \Cref{norm:thm:wreath} then gives \eqref{eq:scaling}. Thus \Cref{cor:dini-nesting,cor:lone-rearrangement} also apply to $K$.
\end{remark}

\section{Orders, rays and nested sets}\label{sec:routing}
We prove the criterion stated in \Cref{thm:main-rays}. Let $\GammaG$ act transitively on the right on an infinite set $X$. We use the simple $\cS$-Schreier graph, omitting loops and identifying parallel edges. Recall that, for an ordering $Y=\{y_1,\ldots,y_m\}$ of a finite nonempty subset of $X$, we set
\begin{equation}\label{eq:linked-demand}
 D_r(Y):=\{y_i:i<r,\ y_i\sim y_j\text{ for some }j\ge r\}\cup\{y_r\}.
\end{equation}

\begin{definition}\label{def:condition}
For an infinite connected graph of bounded degree and a constant $c$, consider the following condition: every finite nonempty set $Y$ admits an ordering such that, for each $r$, there is a ray from each vertex of $D_r(Y)$ and every edge belongs to at most $c$ of these rays. The ordering is the same for every rank; the rays may depend on $Y$ and $r$.
\end{definition}

\begin{lemma}[From Schreier graphs to Cayley graphs]\label{lem:lift}
Suppose $G$ acts transitively on an infinite set and its simple $S$-Schreier graph satisfies the condition in \Cref{def:condition} with constant $c$. Then its simple $S$-Cayley graph satisfies the same condition with constant $c$.
\end{lemma}

\begin{proof}
Fix $x_0\in X$ and put $p(u)=x_0u$. Given a finite nonempty set $F\subset G$, we order $p(F)$ as in \Cref{def:condition} and then order $F$ by listing each fiber consecutively. Suppose the $r$-th point of $F$ maps to $y_\ell$. We claim that $p(D_r(F))\subset D_\ell(p(F))$. The vertices in its fiber map to $y_\ell$, while a vertex in an earlier fiber can belong to $D_r(F)$ only if it has a neighbor in a later fiber.

Now choose the rays in the Schreier graph and a generator label for each oriented edge in them. Following these labels, we lift the ray from $p(u)$ to a path starting at each $u\in D_r(F)$. The lifted path is again a ray, since its projection has no repeated vertices.

Fix a Cayley edge $\{u,v\}$. Its projection is either a loop, which none of the constructed rays uses, or a single edge of the simple Schreier graph. At most $c$ chosen rays in the Schreier graph use that edge, and each uses it once. For each such ray, the orientation and preceding labels determine the starting vertex of a lift through $\{u,v\}$ uniquely, by right cancellation. Hence at most $c$ lifted rays use $\{u,v\}$.
\end{proof}

\begin{proposition}[Boundary bounds from rays]\label{prop:linked-shelling}
Suppose $X$ satisfies the condition in \Cref{def:condition} with an integer constant $c\ge1$. Then every finite nonempty $F\subset\GammaG$ can be ordered so that each initial segment $P$ satisfies
\[
  \edgebd(P)\le(1+\Delta c)\edgebd(F).
\]
Consequently,
\[
  \IexE(n)\le(1+\Delta c)\IupE(n)
  \qquad(n\ge1).
\]
\end{proposition}

\begin{proof}
By \Cref{lem:lift}, the Cayley graph satisfies the condition in \Cref{def:condition}, so we order $F$ accordingly. Let $P$ be the initial segment ending at rank $r$. Every edge from $P$ to $F\setminus P$ starts in $D_r(F)$, and we need to bound the number of these vertices. Assign to each vertex of $D_r(F)$ the last edge of its ray leaving $F$. At most $c$ rays use each edge, which gives $|D_r(F)|\le cB_{\cS}(F)$ and hence $B_{\cS}(P)\le B_{\cS}(F)+\Delta|D_r(F)|$. To obtain the profile bound, we take the first $n$ points of a set attaining $I_{E,\uparrow}(n)$.
\end{proof}

\subsection{Quasi-isometries}
For a graph $X$, let $d_X$ be graph distance, $B_X(x,R)=\{y:d_X(x,y)\le R\}$, and $N_R(A)=\{x:d_X(x,A)\le R\}$, where $d_X(x,A)=\inf_{y\in A}d_X(x,y)$. These neighborhoods are taken in the ambient graph. For bounded-degree $X$, write $V_X(R)=\sup_x|B_X(x,R)|$.

\begin{theorem}[Quasi-isometry invariance]\label{thm:transport}
The condition in \Cref{def:condition} is a quasi-isometry invariant among bounded-degree graphs. The constant in that condition for one graph depends only on the constant for the other graph, the quasi-isometry distortion and the degree bounds.
\end{theorem}

\begin{proof}
Suppose maps $f:V(X)\to V(Z)$ and $g:V(Z)\to V(X)$ satisfy, for integers $a,b\ge1$ and $q\ge0$,
\begin{align*}
 d_Z(fx,fx')&\le a&&\text{when }x\sim x',\\
 d_X(gz,gz')&\le b&&\text{when }z\sim z',\\
 d_X(gfx,x)&\le q,&d_Z(fgz,z)&\le q.
\end{align*}
If $Z$ satisfies the condition in \Cref{def:condition} with constant $c$, we will construct rays in $X$ such that the number of rays containing any edge is bounded in terms of $a,b,q,c$ and the degree bounds.

An infinite walk is \emph{proper} if it visits every finite vertex set only finitely often. How do we obtain the required rays after changing the graph? We first use an ordering in $Z$ to order the finite set in $X$. The corresponding rays in $Z$ give infinite walks in $X$. We will bound the number of these walks containing each edge, check that the walks are proper, and then erase loops to obtain rays.

Let $A\subset V(X)$ be finite and nonempty, and put $B=N_a(f(A))$. We order the finite set $B$ as in \Cref{def:condition} and list the points of $A$ according to the positions of their images in $B$, resolving ties arbitrarily. Write the resulting order as $A=\{x_1,\ldots,x_m\}$. Fix $r$, and let $t$ be the rank of $f(x_r)$ in $B$.

For each $d\in D_r(A)$, choose $z(d)\in D_t(B)$ with
\begin{equation}\label{eq:near-source}
 d_Z(f(d),z(d))\le a.
\end{equation}
If $d=x_r$ or $f(d)=f(x_r)$, use $z(d)=f(x_r)$. Otherwise $d$ precedes $x_r$ and has a neighbor $y\in A$ of rank at least $r$. The $B$-rank of $f(d)$ is less than $t$, and that of $f(y)$ is at least $t$. A shortest path from $f(d)$ to $f(y)$ has length at most $a$ and lies in $B$. At its first edge crossing from ranks less than $t$ to ranks at least $t$, the earlier endpoint belongs to $D_t(B)$ and satisfies \eqref{eq:near-source}.

Set
\[
 R=q+ab,\qquad M=V_X(R),\qquad N=\Delta_Z V_Z(R),
\]
where $\Delta_Z$ is the maximum degree of $Z$. For each $z$, there are at most $M$ vertices $d$ with $z(d)=z$, since
\[
 d_X(d,gz)\le d_X(d,gf(d))+d_X(gf(d),gz)\le q+ba=R.
\]
Choose rays in $Z$ from $D_t(B)$ so that each edge belongs to at most $c$ rays. For each $d$, first join $d$ to $g(z(d))$ by a shortest $X$-path. Then follow the ray from $z(d)$, replacing each unoriented $Z$-edge $uv$ by a fixed shortest $X$-path from $gu$ to $gv$, reversed if necessary. These replacement paths have length at most $b$.

Fix an $X$-edge and one of its endpoints $x$. Every initial connecting path using this edge starts within distance $R$ of $x$, so at most $M$ connecting paths use it. If the replacement of a $Z$-edge $uv$ uses it, then $d_X(gu,x)\le b$, and hence
\[
 d_Z(u,fx)\le d_Z(u,fgu)+d_Z(fgu,fx)\le q+ab=R.
\]
At most $N$ edges of $Z$ have replacement paths that use the fixed $X$-edge. Each such edge belongs to at most $c$ chosen rays, and each ray is used for at most $M$ assigned vertices. Thus at most $McN$ of the walks obtained from rays traverse the fixed edge. The connecting paths add at most $M$, giving the bound $M(1+cN)$.

We must also check that the infinite walks are proper. If a replacement path visits a finite set $K\subset V(X)$, its starting $Z$-vertex lies in the finite set $N_R(f(K))$. A simple $Z$-ray visits each such starting vertex at most once, and each replacement path is finite. Thus the resulting walk visits every finite $X$-set only finitely often, as required.

We can now erase loops in each proper infinite walk. Retain the starting vertex, pass to the vertex after its last occurrence, and repeat; properness makes each step well-defined. The retained vertices form a ray with the same starting vertex and use only edges of the walk, so every edge belongs to at most $M(1+cN)$ of the resulting rays.

A quasi-isometry and a coarse inverse provide the displayed hypotheses with finite integers $a,b,q$. Interchanging $X$ and $Z$ proves the converse.
\end{proof}

\subsection{Nested exhaustions}

For a finite set $A\subset X$, let $B_A(U)=|E_X(U,A\setminus U)|$ for $U\subset A$. We begin with an averaging observation.

\begin{lemma}[Intersections with translates]\label{lem:relative-averaging}
Let $X$ be a Cayley graph, or an infinite connected amenable vertex-transitive graph of finite degree. For nonempty finite sets $A,F\subset X$, some automorphic image $gF$ has nonempty intersection $U=A\cap gF$ satisfying
\[
 \frac{B_A(U)}{|U|}\le\frac{B(F)}{|F|}.
\]
\end{lemma}
\begin{proof}
For a right Cayley graph, sum over left translations. Counting pairs of vertices gives
\[
 \sum_g|A\cap gF|=|A||F|,\qquad
 \sum_g B_A(A\cap gF)\le |A|B(F).
\]
For the second inequality, orient each boundary edge of $F$ from its endpoint in $F$. For each such edge and each $x\in A$, there is a unique translation taking that endpoint to $x$. The requirement that the other endpoint also lie in $A$ can only reduce the count. The two sums are finite, and their ratio proves the assertion.

For an amenable vertex-transitive graph, give its automorphism group the topology of pointwise convergence on vertices. The group is unimodular, meaning that a left Haar measure is also right invariant; see \cite[Introduction and Theorem~1.3]{TesseraTointon2023}. Normalize Haar measure so that vertex stabilizers have measure one. The same counting gives the two displayed formulas with integrals in place of sums, because the set of automorphisms taking a specified vertex to another has measure one. The integrands vanish outside a finite union of compact open cosets, so the same ratio argument applies.
\end{proof}

We now give another sufficient condition for the finite-set orderings used below. The averaging and successive subdivision follow the arguments in \cite[\S\S2.2--2.3]{FreedmanGuthMurphy2015}; we give the details for finite subsets of an infinite Cayley graph.

\begin{theorem}[Orderings from profile summability]\label{thm:dini-orderings}
Let $G$ be infinite and finitely generated, with a fixed finite symmetric generating set. If its monotone edge profile satisfies \eqref{eq:dini}, there is $K<\infty$ such that every finite nonempty $A\subset G$ admits an ordering with
\[
 B(P)\le K B(A)\qquad\text{for every initial segment }P.
\]
\end{theorem}
\begin{proof}
Let $X$ be the Cayley graph of $G$ for the fixed generating set, and let $\Delta$ be its degree. Choose $D$ such that
\[
 \sum_{r\ge0}a(\lfloor m/2^r\rfloor)\le D a(m)\qquad(m\ge1).
\]
An eventual bound extends to every positive integer $m$ by enlarging $D$, since $a(m)\ge1$.

We begin by dividing any finite set $A$ of size $m\ge4$ into two parts. Put $k=\lfloor m/4\rfloor$ and choose $F$ attaining $j(k)$. Apply \Cref{lem:relative-averaging} to the remaining vertices of $A$, remove the resulting nonempty intersection with a translate of $F$, and repeat until the total number $T$ of removed vertices first reaches $\lceil m/4\rceil$. Each removed set $U$ has at most $k$ points and its boundary within the set then remaining is at most $|U|/j(k)$. Therefore
\[
 \lceil m/4\rceil\le T\le\lceil m/4\rceil-1+\lfloor m/4\rfloor\le\lfloor m/2\rfloor.
\]
Let $A_1$ be the union of the removed sets and put $A_2=A\setminus A_1$. Both parts have size at most $\lfloor3m/4\rfloor$. Every edge from $A_1$ to $A_2$ was counted when its endpoint in $A_1$ was removed, whence
\[
 |E_X(A_1,A_2)|\le\frac{T}{j(k)}
 \le\frac{m}{2j(k)}\le4\Delta^2a(m).
\]
The last inequality follows from $m/k\le7$ and \Cref{norm:thm:universal}. For $m=2,3$, an arbitrary division into parts whose sizes differ by at most one has the same properties.

Divide each nonsingleton part in this way, and list all the vertices of its first part before those of its second. For any initial segment, only one part at each stage is partially listed. Its internal crossing edges lie among the edges of the divisions along this succession of parts. Their sizes are at most $\lfloor(3/4)^rm\rfloor$, so the number of these edges is at most
\[
 4\Delta^2\sum_{r\ge0}a(\lfloor(3/4)^rm\rfloor)
 \le12\Delta^2\sum_{q\ge0}a(\lfloor m/2^q\rfloor)
 \le12\Delta^2D a(m).
\]
For the middle inequality, write $r=3q+s$, where $0\le s<3$, and use $(3/4)^3<1/2$. If $P$ is an initial segment, the edges from $P$ to $X\setminus A$ contribute at most $B(A)$. Since $a(m)\le B(A)$, we conclude that
\[
 B(P)\le B(A)+12\Delta^2D a(m)\le(1+12\Delta^2D)B(A).
\]
\end{proof}

\begin{theorem}\label{thm:record-stacking}
Suppose there is $K<\infty$ such that every finite nonempty $A\subset\GammaG$ can be ordered so that each initial segment $P$ satisfies $\edgebd(P)\le K\edgebd(A)$. Then there are $C=C(\GammaG,\cS,K)$, $n_0$, and nested sets $(W_n)_{n\ge n_0}$ with $\card{W_n}=n$, $\bigcup_nW_n=\GammaG$, and
\[
  \edgebd(W_n)\le C\IupE(n)\qquad(n\ge n_0).
\]
If $\GammaG$ is amenable, $(W_n)$ is a F{\o}lner sequence.
\end{theorem}

\begin{proof}
Write $a(n)=\IupE(n)$. If $a$ is bounded, \Cref{lem:bounded-profile-nesting} gives an enumeration with uniformly bounded edge cuts. Since $a(n)\ge1$, its initial segments satisfy the required estimate. We may therefore suppose that $a(n)\to\infty$.

To obtain a nested exhaustion from the orderings of individual finite sets, we choose minimizing sets at successive doubling thresholds of the monotone profile, so that the sum of their boundary sizes is at most a constant times the last boundary size. We then add disjoint translates of these sets one point at a time, following their chosen orders. To ensure that the union is the whole group, we also insert at each stage the first missing element of a fixed enumeration.

Then $0\le a(n+1)-a(n)\le\Delta$: a set attaining $a(n)$ either has at least $n+1$ points or can be enlarged by one point, using \Cref{lem:boundary-editing}. Choose $j_0$ with $2^{j_0}>a(1)$ and put
\[
  r_j=\min\{n:a(n)\ge2^j\}\qquad(j\ge j_0).
\]
These indices exist since $a(n)\to\infty$. Choose $A_j$ with $\card{A_j}=m_j\ge r_j$ and $\edgebd(A_j)=a(r_j)=:b_j$. By the bound on $a(n+1)-a(n)$ and the definition of $r_j$,
\[
  2^j\le b_j\le2^j+\Delta,\qquad a(m_j)=b_j.
\]

Fix an enumeration $\GammaG=\{\gamma_1,\gamma_2,\ldots\}$ and choose the assumed orderings of the $A_j$. Set $n_0=m_{j_0}$ and $W_{n_0}=A_{j_0}$. At stage $j+1$, first add the first $\gamma_i$ not yet included. Then add a left translate of $A_{j+1}$ one point at a time in its chosen order. Choose the translate disjoint from the existing set; only finitely many translations are excluded.

Let us check the boundary during this construction. Left translation preserves edge cuts, so each initial segment of the new translate has edge cut at most $Kb_{j+1}$, while each separately inserted point adds at most $\Delta$. The edge cut of a union is at most the sum of the edge cuts. Thus, throughout stage $j+1$, we have
\[
  \edgebd(W_n)\le
  \sum_{i=j_0}^{j}b_i+Kb_{j+1}+\Delta(j-j_0+1)
  \le Cb_j.
\]
At this stage $n\ge m_j$, so $a(n)\ge a(m_j)=b_j$, and the desired bound follows. It also holds at $n_0$. Since every $\gamma_i$ is eventually included, we have $\bigcup_nW_n=\GammaG$. We may order the initial finite set $A_{j_0}$ arbitrarily to obtain an enumeration of $\GammaG$, since the stated bound is required only for $n\ge n_0$.

If $\GammaG$ is amenable, $a(n)/n\le\Delta\Iex(n)/n\to0$ by \Cref{lem:subadditivity-amenability}, so $(W_n)$ is F{\o}lner. Also $\card{\bd W_n}\le C\Delta\Iup(n)$ by \eqref{eq:profile-comparison}.
\end{proof}

\begin{corollary}[Nested sets from summability]\label{cor:dini-nesting}
If $G$ satisfies \eqref{eq:dini}, it has a nested exhaustion $(W_n)_{n\ge n_0}$ with $|W_n|=n$ and $B(W_n)\le C a(n)$. This applies to every direct product in \Cref{prop:mixed-products} and every wreath product in \Cref{norm:thm:wreath}. In the amenable cases, these exhaustions are F{\o}lner sequences.
\end{corollary}
\begin{proof}
Apply \Cref{thm:dini-orderings,thm:record-stacking}. The product estimate and the proof of \Cref{norm:thm:wreath} give \eqref{eq:scaling}, hence \eqref{eq:dini}. For a nonamenable group, $a(n)\asymp n$, so the same summability holds.
\end{proof}

The next consequence is the $L^1$ rearrangement argument of Steinerberger \cite{Steinerberger2024}. We include the coarea proof to specify the rearrangement furnished by our exhaustions.

\begin{corollary}[$L^1$ rearrangement]\label{cor:lone-rearrangement}
Fix a Cayley graph $X$ of an infinite finitely generated group $G$. Suppose $G$ has a nested exhaustion with $|W_n|=n$ and $B(W_n)\le C a(n)$ for all sufficiently large $n$. There is an enumeration $G=\{g_1,g_2,\ldots\}$ such that, if $f^*$ places the values of $|f|$ in nonincreasing order along this enumeration, then
\[
 \sum_{\{x,y\}\in E_X}|f^*(x)-f^*(y)|
 \le C'\sum_{\{x,y\}\in E_X}|f(x)-f(y)|
\]
for every finitely supported real or complex function $f$ on $G$.
\end{corollary}
\begin{proof}
Enumerate the first set arbitrarily and then list the unique point added at each step. Enlarge the constant to include the finitely many initial cardinalities; this is possible because $I_E(n)\ge1$. Every initial segment of length $n$ then has boundary at most $C'I_E(n)$. The superlevel sets of $f^*$ are these initial segments, so edge coarea gives
\[
 \begin{split}
 \sum_{\{x,y\}\in E_X}|f^*(x)-f^*(y)|
 &=\int_0^\infty B(\{f^*>t\})\,dt\\
 &\le C'\int_0^\infty B(\{|f|>t\})\,dt
 \le C'\sum_{\{x,y\}\in E_X}|f(x)-f(y)|.
 \end{split}
\]
\end{proof}

\begin{theorem}[Schreier graphs]\label{thm:schreier}
Let $G$ be infinite and finitely generated. Suppose it has an infinite transitive Schreier graph quasi-isometric to the Cayley graph of a group $H$ satisfying the condition in \Cref{def:condition}. Then, for every finite symmetric generating set $S$ of $G$, there is $K<\infty$ such that every finite $F\subset G$ has an order with $B_S(P)\le K B_S(F)$ for every initial segment $P$, and
\[
 \IexE^{G,S}(n)\le K\IupE^{G,S}(n)\qquad(n\ge1).
\]
There is also a nested exhaustion as in \Cref{thm:record-stacking}.
\end{theorem}

\begin{proof}
By \Cref{thm:transport}, the Schreier graph satisfies the condition in \Cref{def:condition}, and we can apply \Cref{prop:linked-shelling,thm:record-stacking}. It remains only to note that changing the generating set changes distances in the Schreier graph by at most fixed factors. A further application of \Cref{thm:transport} therefore gives the assertion for every $S$.
\end{proof}

Related constructions use edge-disjoint paths between an earlier initial segment and the complement of a later one in finite graphs \cite[Definition~4 and Lemma~3]{GiannopoulouEtAl}. In infinite graphs, related constructions start edge-disjoint rays with the boundary edges of a subgraph and continue within it \cite[Theorem~1]{Thomassen2016}. The condition in \Cref{def:condition} requires a single ordering of each finite set, with each edge belonging to at most $c$ rays for a constant $c$ independent of the rank and the set.

\section{Actions with a height function}\label{sec:stacking}

\subsection{Homomorphisms to the integers}

Let $\pi\colon\GammaG\to\Z$ be a nonzero homomorphism. After dividing by a positive integer, we may assume that $\pi$ is onto. Put $L=\max_{s\in\cS}|\pi(s)|$. The order induced by $\pi$ controls both right and left edge cuts; see \cite{Gromov2008} for related slicing arguments.

\begin{proposition}[Homomorphisms to $\Z$]\label{prop:indicable-shelling}
Order a finite nonempty $F\subset\GammaG$ by nondecreasing $\pi$-values, with ties arbitrary. Each initial segment $F_r$ satisfies
\[
  \edgebd(F_r)\le C_\pi\edgebd(F),
  \qquad
  B^{\mathrm L}_{\cS}(F_r)\le C_\pi B^{\mathrm L}_{\cS}(F),
  \qquad
  C_\pi=1+\frac{\Delta(L+1)}2.
\]
Consequently, $\IexE(n)\le C_\pi\IupE(n)$ for every $n\ge1$.
\end{proposition}

\begin{proof}
We first bound the number of points at a single height. Put $M(F)=\max_k\card{F\cap\pi^{-1}(k)}$ and choose $a\in\cS$ with $\pi(a)\ne0$. Distinct points with the same $\pi$-value lie in distinct cosets $x\langle a\rangle$. In each nonempty intersection $F\cap x\langle a\rangle$, the smallest and largest exponents give two distinct boundary pairs, labelled $a^{-1}$ and $a$. Repeating the argument with $\langle a\rangle x$ gives the corresponding bound for the left cut, so
\[
  2M(F)\le\min\{\edgebd(F),B^{\mathrm L}_{\cS}(F)\}.
\]
Let $h=\max_{x\in F_r}\pi(x)$. A right edge from $F_r$ to $F\setminus F_r$ starts at a point $x$ with $h-L\le\pi(x)\le h$. The same holds for left edges because $\pi(sx)=\pi(s)+\pi(x)$. There are at most $(L+1)M(F)$ such points, each incident to at most $\Delta$ crossing edges. Hence
\[
  \edgebd(F_r)\le\edgebd(F)+\Delta(L+1)M(F),
  \qquad
  B^{\mathrm L}_{\cS}(F_r)\le B^{\mathrm L}_{\cS}(F)+\Delta(L+1)M(F).
\]
The estimate for $M(F)$ gives both bounds. Taking the first $n$ points of a set attaining $\IupE(n)$ proves the profile comparison.
\end{proof}

\subsection{Ordering by height}

Let $\GammaG$ act transitively on the right on an infinite set $X$, and suppose $h\colon X\to\Z$ satisfies
\begin{equation}\label{eq:schreier-height-data}
W:=\sup_{j\in\Z}\card{h^{-1}(j)}<\infty,
\qquad
L:=\sup_{x\in X,\,s\in\cS}\card{h(xs)-h(x)}<\infty.
\end{equation}

\begin{theorem}[Height functions]
\label{thm:schreier-line-ratio}
Suppose \eqref{eq:schreier-height-data} holds. Every finite nonempty $F\subset\GammaG$ can be ordered so that its first $r$ points $F_r$ satisfy

$$
  \edgebd(F_r)\le\bigl(1+\Delta W(L+1)\bigr)\edgebd(F).
$$

Consequently,

$$
  \IexE(n)\le\bigl(1+\Delta W(L+1)\bigr)\IupE(n)
  \qquad(n\ge1).
$$

\end{theorem}

\begin{proof}
We order a finite nonempty $Y\subset X$ by nondecreasing height, placing points of equal height in any order, and put $\ell=h(y_r)$. If $i<r$, $j\ge r$ and $y_i\sim y_j$, then $h(y_i)\le\ell\le h(y_j)$ and $h(y_j)-h(y_i)\le L$. Thus the earlier vertices adjacent to later ones lie in a bounded interval of heights. Including $y_r$, we obtain

$$
  D_r(Y)\subseteq h^{-1}([\ell-L,\ell]),
  \qquad
  \card{D_r(Y)}\le W(L+1).
$$

Since $X$ is infinite, connected and locally finite, each vertex lies on a ray. We simply choose one from each point of $D_r(Y)$. Each edge then belongs to at most $|D_r(Y)|\le W(L+1)$ of the chosen rays, so \Cref{prop:linked-shelling} applies.
\end{proof}

\begin{corollary}[Schreier graphs isomorphic to a line]\label{cor:schreier-line-sharp}
If the simple $\cS$-Schreier graph is a bi-infinite path, then

$$
  \IexE(n)\le(1+\Delta)\IupE(n)
  \qquad(n\ge1).
$$

\end{corollary}

\begin{proof}
Order each finite nonempty $Y$ from left to right. Then $D_r(Y)$ consists of $y_r$ and, if it belongs to $Y$, the vertex immediately to its left. Choose the ray from $y_r$ going right and, when present, the ray from its left neighbor going left. They share no edge, so \Cref{prop:linked-shelling} applies with $c=1$.
\end{proof}

\begin{lemma}[Nested sets from a height function]\label{lem:height-fibres}
Under the hypotheses of \Cref{thm:schreier-line-ratio}, there are $C<\infty$, $n_0$, and an enumeration $\GammaG=\{g_1,g_2,\ldots\}$ such that
\[
  B_{\cS}(\{g_1,\ldots,g_n\})\le C\IupE(n)\qquad(n\ge n_0).
\]
If $\GammaG$ is amenable, these initial segments form a F{\o}lner sequence.
\end{lemma}

\begin{proof}
Apply \Cref{thm:schreier-line-ratio,thm:record-stacking}.
\end{proof}

\subsection{Finite-index subgroups and quotients}

\begin{lemma}[Finite-index extension]\label{lem:finite-index-thickening}
Let $H\le G$ have finite index, and choose representatives $Q$ so that $G=\bigsqcup_{q\in Q}Hq$. For finite $E\subset H$, put $\widetilde E=EQ$. Then $\card{\widetilde E}=\card Q\card E$. For fixed finite symmetric generating sets $\cT$ of $H$ and $\cS$ of $G$, there is $C=C(H,G,\cT,\cS,Q)$ such that
\begin{equation}\label{eq:finite-index-boundary-upper}
  B_{\cS}(\widetilde E)\le C B_{\cT}(E).
\end{equation}
Moreover, $\card{E'Q\setminus EQ}=\card Q\card{E'\setminus E}$ whenever $E\subset E'\subset H$.
\end{lemma}

\begin{proof}
For $q\in Q$ and $s\in\cS$, write $qs=h(q,s)q'$ with $h(q,s)\in H$ and $q'\in Q$. The $s$-edge from $xq$, with $x\in E$, leaves $EQ$ exactly when $xh(q,s)\notin E$. Since $\card{Eh\setminus E}=\frac12\card{Eh\triangle E}$, \eqref{eq:word-symdiff} gives
\[
  B_{\cS}(EQ)
  =\sum_{q\in Q,\,s\in\cS}\card{Eh(q,s)\setminus E}
  \le\left(\sum_{q\in Q,\,s\in\cS}\card{h(q,s)}_{\cT}\right)B_{\cT}(E).
\]
The cardinality identities follow because the cosets $Hq$ are disjoint.
\end{proof}

\begin{lemma}[Restriction to a coset]\label{lem:coset-restriction}
Let $H\le\GammaG$ have finite index, and choose representatives $Q$ for the cosets $Hq$. Fix finite symmetric generating sets $\cT$ of $H$ and $\cS$ of $\GammaG$. Every finite nonempty $A\subset\GammaG$ has a coset $Hq$ with $\card{A\cap Hq}\ge\card A/\card Q$ such that $E=(A\cap Hq)q^{-1}$ satisfies
\[
  B_{\cT}(E)\le C B_{\cS}(A),
\]
where $C$ depends only on $H,\GammaG,\cT,\cS,Q$.
\end{lemma}

\begin{proof}
Choose $q$ for which $\card{A\cap Hq}$ is largest. For $\tau\in\cT$, put $w_\tau=q^{-1}\tau q$. Since $(Hq)w_\tau=Hq$, \eqref{eq:word-symdiff} gives
\[
  \card{E\tau\triangle E}
  =\card{(Aw_\tau\triangle A)\cap Hq}
  \le2\card{w_\tau}_{\cS}B_{\cS}(A).
\]
Summing over $\tau$ and using \eqref{eq:right-symdiff} gives the required estimate. The constant is uniform because $Q$ and $\cT$ are finite.
\end{proof}

\begin{proposition}[Finite-index subgroups]\label{prop:finite-index-ascent}
Let $H\le G$ have finite index. If $I/I_{\uparrow}$ is bounded for $H$, it is bounded for $G$. Fix finite symmetric generating sets $\cT$ of $H$ and $\cS$ of $G$. If $H$ has nested sets $(W_n)_{n\ge n_0}$ with $\bigcup_nW_n=H$ and
\[
  \card{W_n}=n,
  \qquad B_{\cT}(W_n)\le A I_{E,\uparrow}^{H,\cT}(n)\quad(n\ge n_0),
\]
then $G$ has such a family for $(G,\cS)$ at all sufficiently large cardinalities. The F{\o}lner property is preserved.
\end{proposition}

\begin{proof}
By \eqref{eq:profile-comparison}, it suffices to use edge profiles. Put $d=[G:H]$ and choose representatives $Q$ with $G=\bigsqcup_{q\in Q}Hq$. Applying \Cref{lem:coset-restriction} to a set attaining $I_{E,\uparrow}^{G,\cS}(r)$ gives
\begin{equation}\label{eq:finite-index-minorant-down}
  I_{E,\uparrow}^{H,\cT}(\lceil r/d\rceil)
  \le C_0 I_{E,\uparrow}^{G,\cS}(r).
\end{equation}
Now put $k=\lceil r/d\rceil$ and choose a minimizing $k$-point set $E\subset H$. We form $EQ$ and delete $dk-r<d$ points to reach the prescribed cardinality. Using \Cref{lem:finite-index-thickening,lem:boundary-editing}, the assumed bound in $H$, and \eqref{eq:finite-index-minorant-down}, we obtain
\[
  I_E^{G,\cS}(r)
  \le C I_E^{H,\cT}(k)+\Delta(d-1)
  \le C_1 I_{E,\uparrow}^{G,\cS}(r),
\]
since $I_{E,\uparrow}^{G,\cS}\ge1$.

For the nested family, add each difference $W_{n+1}Q\setminus W_nQ$ one point at a time. The resulting sets $U_r$ exhaust $G$, have $\card{U_r}=r$, and, for $dn\le r<d(n+1)$, satisfy
\[
  B_{\cS}(U_r)\le C B_{\cT}(W_n)+\Delta(d-1)
  \le C_2 I_{E,\uparrow}^{G,\cS}(r).
\]
For the second inequality, use \eqref{eq:finite-index-minorant-down} at $dn$ and monotonicity. If $(W_n)$ is F{\o}lner, the first inequality and $r\ge dn$ give $B_{\cS}(U_r)/r\to0$.
\end{proof}

\begin{corollary}[Branch groups]\label{cor:branch-groups}
Every finitely generated branch group satisfies Gromov's comparison and admits a nested exhaustion with $|W_n|=n$ and $B(W_n)\le C a(n)$ for all sufficiently large $n$. Its Cayley graph admits the rearrangement in \Cref{cor:lone-rearrangement}. In the amenable case, the exhaustion is F{\o}lner.
\end{corollary}
\begin{proof}
Let $G$ have a faithful branch action on a locally finite rooted tree. Thus $G$ acts transitively on every level. For a vertex $v$, the rigid stabilizer $\operatorname{rist}_G(v)$ is the subgroup fixing every vertex outside the subtree rooted at $v$. For a level $L$ containing at least two vertices, the level rigid stabilizer
\[
 \operatorname{rist}_G(L)=\prod_{v\in L}\operatorname{rist}_G(v)
\]
has finite index in $G$; see \cite[p.~33]{GarridoWilson2014}. It is finitely generated, and each direct factor is finitely generated as a quotient. Level transitivity makes the factors conjugate. If one were finite, they would all be finite, contradicting infinitude of $G$. Thus \Cref{prop:mixed-products,cor:dini-nesting} apply to this subgroup. Now use \Cref{prop:finite-index-ascent,cor:lone-rearrangement}.
\end{proof}

\begin{corollary}[Finite bases in wreath products]\label{cor:finite-base-wreath}
Let $\Gamma$ be an infinite finitely generated group and let $H$ be finite with $|H|\ge2$. Then $\Gamma\wr H$ satisfies Gromov's comparison and has a nested exhaustion with boundary bounded by a constant multiple of its monotone edge profile.
\end{corollary}
\begin{proof}
The base subgroup $\Gamma^H$ has finite index and is a direct product of $|H|$ infinite finitely generated groups. Apply \Cref{prop:mixed-products,cor:dini-nesting,prop:finite-index-ascent}.
\end{proof}

Together with \Cref{norm:thm:wreath,cor:dini-nesting}, this proves \Cref{thm:main-wreath}.

\begin{proposition}[Quotients]\label{prop:line-quotient}
Let $p\colon G\to Q$ be a surjective homomorphism of finitely generated groups. Suppose $Q$ has a transitive action on an infinite set $X$ and a map $h:X\to\Z$ satisfying \eqref{eq:schreier-height-data} for some finite symmetric generating set. Then $G$ has bounded $\Iex/\Iup$ and nested sets exhausting $G$ with $\card{W_n}=n$ and $\edgebd(W_n)\le C\IupE(n)$ for all sufficiently large $n$.
\end{proposition}

\begin{proof}
Let $G$ act by $x\cdot g=xp(g)$. If the height changes by at most $L_Q$ under the generators $\cT$ of $Q$, then for any finite symmetric generating set $\cS$ of $G$,
\[
  \card{h(xp(s))-h(x)}\le L_Q\card{p(s)}_{\cT}\qquad(x\in X,\ s\in\cS).
\]
Thus the induced action satisfies \eqref{eq:schreier-height-data}, and \Cref{thm:schreier-line-ratio,lem:height-fibres} apply.
\end{proof}

\subsection{Applications}

\begin{corollary}[Virtually indicable groups]\label{cor:virtually-indicable}
Let $\GammaG$ be an infinite finitely generated group with a finite-index subgroup admitting a nonzero homomorphism to $\Z$. Then $\Iex/\Iup$ is bounded. There are nested sets $W_n$ exhausting $\GammaG$ with $\card{W_n}=n$ and $\edgebd(W_n)\le C\IupE(n)$ for all large $n$. If $\GammaG$ is amenable, these sets form a F{\o}lner sequence.
\end{corollary}

\begin{proof}
After normalizing the homomorphism $\pi$ to be onto, let the subgroup act on $\Z$ by $j\cdot g=j+\pi(g)$ and take $h(j)=j$. We have $W=1$ and $L<\infty$, so we can apply \Cref{thm:schreier-line-ratio,lem:height-fibres} and then use \Cref{prop:finite-index-ascent} to pass to $\GammaG$.
\end{proof}

\begin{corollary}[Elementary amenable groups]\label{cor:solvable}
Every infinite finitely generated elementary amenable group is virtually indicable and satisfies \Cref{cor:virtually-indicable}. This includes all infinite finitely generated virtually solvable groups. The same conclusions hold for finitely generated groups with infinite abelianization and finitely generated HNN extensions.
\end{corollary}

\begin{proof}
An infinite finitely generated elementary amenable group has a finite-index subgroup mapping onto $\Z$; see \cite[Lemma~1.7]{Hillman2002}. A finitely generated group with infinite abelianization also maps onto $\Z$. For an HNN extension, send the base group to $0$ and the stable letter to $1$.
\end{proof}

\begin{corollary}[The Grigorchuk family]\label{cor:grigorchuk-ratio}
Let $G_\omega=\langle a,b_\omega,c_\omega,d_\omega\rangle$, $\omega\in\{0,1,2\}^{\N}$, be the Grigorchuk group with its standard action on the rooted binary tree \cite[Definition~3.1.1]{Schneeberger2024}, and let $\cS_\omega$ be this generating set with identities and repetitions removed. There is an absolute constant $C_{\mathrm{Gr}}$, uniform over $\omega$, such that
\[
 \IexE(n)\le C_{\mathrm{Gr}}\IupE(n),\qquad
 \Iex(n)\le4C_{\mathrm{Gr}}\Iup(n)\qquad(n\ge1).
\]
For every finite generating set, $\Iex/\Iup$ is bounded and there are nested F{\o}lner sets $W_n$ exhausting $G_\omega$ with $\card{W_n}=n$ and $\edgebd(W_n)\le C\IupE(n)$ for all large $n$.
\end{corollary}

\begin{proof}
We identify the boundary of the rooted binary tree with $\{0,1\}^{\N}$, and write $w^\infty$ for the infinite repetition of a nonempty finite word $w$. The orbit of $0^\infty$ consists of binary sequences with finitely many $1$'s \cite[Definition~3.1.1 and Proposition~3.2.1]{Schneeberger2024}. At each such sequence, $a$ flips the first digit; one of $b_\omega,c_\omega,d_\omega$ fixes the sequence, and the other two flip the digit after the first $0$. Thus every vertex has two distinct neighbors. The simple Schreier graph is infinite and connected, hence a bi-infinite path. Every generator equals its inverse, so $x\cdot g=g^{-1}x$ gives the same graph. Since $\card{\cS_\omega}\le4$, \Cref{cor:schreier-line-sharp,eq:profile-comparison} give the two profile estimates with $C_{\mathrm{Gr}}=5$. The nested family follows from \Cref{lem:height-fibres}. Both conclusions hold for every finite generating set by \eqref{eq:word-symdiff}.

The groups $G_\omega$ have intermediate growth unless $\omega$ is eventually constant, when they are virtually abelian \cite{Grigorchuk1984}. Thus all are amenable.
\end{proof}

A group is called FAb if every finite-index subgroup has finite abelianization.

Let $\Omega_0$ consist of sequences in which $0,1,2$ each occur infinitely often. For $\omega\in\Omega_0$, $G_\omega$ is torsion \cite{Grigorchuk1984}. Every finite-index subgroup is then finitely generated and torsion, so has finite abelianization; thus $G_\omega$ is FAb. The first Grigorchuk group has $\omega=(012)^\infty$.

By \Cref{prop:line-quotient,prop:finite-index-ascent}, the profile comparison and the nested-exhaustion conclusion extend to every finitely generated group having a finite-index subgroup which maps onto some $G_\omega$. The nested sets are F{\o}lner when the group is amenable. For $\omega\in\Omega_0$, the restricted wreath product $(\Z/2\Z)\wr G_\omega$ is finitely generated, amenable, torsion, and of exponential growth. It is therefore FAb, and its natural projection onto $G_\omega$ gives such a nested F{\o}lner sequence.

For a Cantor homeomorphism $T$, write $[[T]]$ for its topological full group.

\begin{corollary}[Simple subgroups of topological full groups]\label{cor:full-groups}
Let $T$ be a minimal homeomorphism of a Cantor space and let $\Lambda\le[[T]]$ be an infinite finitely generated simple subgroup. For every finite generating set, $\Iex/\Iup$ is bounded and there are nested F{\o}lner sets $W_n$ exhausting $\Lambda$ with $\card{W_n}=n$ and $\edgebd(W_n)\le C\IupE(n)$ for all large $n$. This includes commutator subgroups of topological full groups of infinite minimal subshifts and Nekrashevych's simple periodic groups of intermediate growth.
\end{corollary}

\begin{proof}
Identifying a $T$-orbit with $\Z$ gives a faithful action of $[[T]]$ by permutations $g$ with $\sup_j\card{g(j)-j}<\infty$ \cite[p.~776]{JuschenkoMonod2013}. We give $\Lambda$ the right action $j\cdot g=g^{-1}(j)$. We first need an infinite orbit. If all its orbits were finite, each orbit action would be trivial, because an infinite simple group has no nontrivial homomorphism to a finite group; this would contradict faithfulness. On an infinite orbit $X\subset\Z$, we take $h(j)=j$. Then $W=1$, and the finite generating set gives $L<\infty$, so \Cref{thm:schreier-line-ratio,lem:height-fibres} apply. Finally, $[[T]]$ is amenable \cite[Theorem~A]{JuschenkoMonod2013}, so $\Lambda$ is amenable and the family is F{\o}lner.

For an infinite minimal subshift, the commutator subgroup of its topological full group is infinite, simple, and finitely generated \cite[Theorems~4.9 and~5.4]{Matui2006}. Nekrashevych's simple periodic groups of intermediate growth embed in topological full groups of minimal $\Z$-actions \cite[Theorems~1.1--1.2 and Proposition~5.1]{Nekrashevych2018}.
\end{proof}

\begin{corollary}[Possible counterexamples]\label{cor:fab}
A finitely generated group is FAb if and only if it is not virtually indicable. Any counterexample to Gromov's question is amenable and FAb. None of its finite-index subgroups admits an action as in \Cref{thm:schreier-line-ratio}, or maps onto a group admitting one.
\end{corollary}

\begin{proof}
Finite-index subgroups are finitely generated, and a finitely generated group maps onto $\Z$ exactly when its abelianization is infinite. Apply \Cref{thm:schreier-line-ratio,prop:finite-index-ascent,prop:line-quotient} and \eqref{eq:nonamenable-ratio}.
\end{proof}

For each fixed generating set, a counterexample has $\IupE(v)\le2\IupE(u)$ for positive integers $u\le v$ with $v/u$ arbitrarily large, by \Cref{thm:records} and \eqref{eq:profile-comparison}.

\begin{lemma}[Property FW and height functions]\label{lem:fw-height}
Let $G$ be a finitely generated group with Property FW. Neither $G$ nor any quotient of a finite-index subgroup of $G$ admits an infinite transitive action with a height satisfying \eqref{eq:schreier-height-data}.
\end{lemma}

\begin{proof}
Property FW passes to finite-index subgroups and quotients \cite[Section~5.B]{Cornulier2016}, so it suffices to consider $G$. Suppose such an action exists, and let $\delta_X E$ denote the unoriented edges leaving $E$ in its simple Schreier graph. The finite sets $E_R=h^{-1}([-R,R])$ exhaust $X$ and satisfy
\[
  \card{\delta_X E_R}\le M:=2\Delta LW.
\]
Every connected component of $E_R$ has an edge leaving $E_R$, so there are at most $M$ components. Choose components $C_R$ with $\card{C_R}\to\infty$; their edge boundaries also have size at most $M$.

We want both colours to have an infinite connected component in the limit, so we choose the roots on edges leaving $C_R$. Take an edge $x_Ry_R$ joining $C_R$ to an infinite component of $X\setminus C_R$, root $X$ at $x_R$, and colour $C_R$ and its complement differently. We then pass to a pointed labelled limit, retaining the colours and the chosen edge: for each fixed radius, the rooted balls are eventually isomorphic, preserving generator labels, vertex colours and the chosen edge. The underlying limit is a Schreier graph of $G$, by compactness of the space of subgroups \cite[Section~2.1]{MatteBonNekrashevychZheng2026}. Bounded degree and $\card{C_R}\to\infty$ imply that the maximum length of a simple path from $x_R$ within $C_R$ tends to infinity. From $y_R$ there are arbitrarily long simple paths within its infinite component. Thus each endpoint of the limiting edge belongs to an infinite connected component of its own colour. At most $M$ edges join the colours. Deleting these edges leaves at least two infinite components, contrary to Property FW \cite[Sections~3.B and~5.A]{Cornulier2016}.
\end{proof}

\section{The Sierpi\'nski-carpet group}\label{sec:carpet}

The preceding argument shows that Property FW excludes infinite transitive actions with a height function satisfying \eqref{eq:schreier-height-data}. We now consider the Sierpi\'nski-carpet group $G_{\mathrm{SC}}$ of Matte Bon--Nekrashevych--Zheng, with its canonical generating set $\cS_{\mathrm{SC}}=\{a,b,c,d\}$ of four involutions \cite{MatteBonNekrashevychZheng2026}. Their Proposition~6.4 shows that $G_{\mathrm{SC}}$ is a branch group, so \Cref{cor:branch-groups} already gives comparison and a nested exhaustion. The geometric argument below proves the condition on rays in \Cref{def:condition} for a Schreier graph of $G_{\mathrm{SC}}$. This stronger conclusion will also be used for alternating full groups in \Cref{sec:fullgroups}.

\begin{theorem}[The Sierpi\'nski-carpet group]\label{thmB}
For the canonical generating set $\cS_{\mathrm{SC}}$, the simple Cayley graph of $G_{\mathrm{SC}}$ satisfies the condition in \Cref{def:condition} with an absolute integer $c_*\ge1$. Consequently, there is an absolute constant $C_{\mathrm{SC}}<\infty$ such that, for every $n\ge1$,
\[
 I_E^{G_{\mathrm{SC}},\cS_{\mathrm{SC}}}(n)
 \le C_{\mathrm{SC}} I_{E,\uparrow}^{G_{\mathrm{SC}},\cS_{\mathrm{SC}}}(n),
 \qquad
 I^{G_{\mathrm{SC}},\cS_{\mathrm{SC}}}(n)
 \le4C_{\mathrm{SC}} I_{\uparrow}^{G_{\mathrm{SC}},\cS_{\mathrm{SC}}}(n).
\]
For every finite symmetric generating set $\cS$, there are $C=C(\cS)<\infty$, $n_0$, and nested F{\o}lner sets $W_n$ exhausting $G_{\mathrm{SC}}$ such that
\[
 |W_n|=n,\qquad B_{\cS}(W_n)\le C I_{E,\uparrow}^{G_{\mathrm{SC}},\cS}(n)
 \qquad(n\ge n_0).
\]
In particular, Gromov's comparison holds for every finite generating set.
\end{theorem}

In a finite approximating graph $X_k$, paths to a distant set $T_k$ will serve in place of rays. Every such path starting in $D_r(V(X_k))\cap A$, with $A$ disjoint from $T_k$, must cross $\delta_{X_k}A$. Thus the cut inequality below is necessary if each edge is to belong to at most $c$ paths. The integral max-flow/min-cut theorem gives sufficiency, and letting $T_k$ recede from the root gives rays after passage to a limit.

\begin{proposition}[Finite approximations]\label{prop:finite-schreier}
Let $(X_k,o_k)$ be finite connected rooted $\cS$-Schreier graphs of $\GammaG$, with $\varnothing\ne T_k\subseteq V(X_k)$ such that $d(o_k,T_k)\to\infty$. Suppose there is an integer $c\ge1$ and an order on each $V(X_k)$ for which
\[
 \card{D_r(V(X_k))\cap A}\le c\card{\delta_{X_k}A}
 \qquad(A\subseteq V(X_k)\setminus T_k).
\]
Here $\delta_{X_k}A$ is the unoriented edge cut in the simple graph, and the inequality holds for every $r$. Then an infinite transitive Schreier graph satisfies the condition in \Cref{def:condition} with constant $c$. Every finite nonempty $F\subset\GammaG$ has an order whose initial segments satisfy $\edgebd(F_r)\le(1+\Delta c)\edgebd(F)$. Consequently, $\IexE(n)\le(1+\Delta c)\IupE(n)$ for all $n$.
\end{proposition}

\begin{proof}
By the integral max-flow/min-cut theorem, we obtain paths from $D_r(V(X_k))$ to $T_k$ such that each edge belongs to at most $c$ paths. Here we cancel opposite flows on each edge before decomposing the flow, and then delete cycles. Points already in $T_k$ use paths of length zero.

We now pass to a pointed labelled limit $X$, which is an infinite Schreier graph of $\GammaG$. Fix a finite $Y\subset X$, restrict the orders to its copies in the finite graphs, and pass to a subsequence on which the restricted order is fixed. At the position of $y_r$ in the full order, $D_r(Y)$ is contained in the corresponding set for that order, so we retain the paths from these vertices. Their lengths tend to infinity because $d(Y,T_k)\to\infty$. Bounded degree and a diagonal argument give a subsequence converging to rays from $D_r(Y)$, with each edge belonging to at most $c$ rays. We can therefore apply \Cref{prop:linked-shelling}.
\end{proof}

We apply \Cref{prop:finite-schreier} to the finite carpet graphs. Write $u_i,v_i$ for the corresponding ternary digits of $u,v$, padded with leading zeros to $k$ places. Let
\begin{equation}\label{eq:finite-carpet}
 \mathcal C_k:=\{(u,v)\in\{0,\ldots,3^k-1\}^2:
 \text{no ternary place has }u_i=v_i=1\},
\end{equation}
with nearest-neighbor edges. Let
\[
  T_k:=\{(u,v)\in\mathcal C_k:
  u\in\{0,3^k-1\}\text{ or }v\in\{0,3^k-1\}\}.
\]
A \emph{standard $\ell$-cell} is a copy of $\mathcal C_\ell$ in the subdivision of $\mathcal C_k$, aligned with the coordinate axes. For a subgraph $H$ and $A\subseteq V(H)$, write $\delta_HA$ for the set of unoriented edges of $H$ with exactly one endpoint in $A$.

Intuitively, the carpet is a landmass with lakes on many scales. We order its vertices by repeatedly dividing rectangles in two, placing the dividing lines where the lakes leave few connections between the two sides. The boundary of each initial segment then meets any standard $\ell$-cell in $O(2^\ell)$ edges, matching the lower bound for a cut when both parts occupy a fixed positive fraction of the cell. We then cover $A\subseteq\mathcal C_k\setminus T_k$ by standard cells $Q_i$, of levels $\ell_i$, for which $\sum_i2^{\ell_i}$ is bounded by a constant times $|\delta_{\mathcal C_k}A|$. Summing the first estimate over this cover will give the required cut inequality.

\begin{lemma}[Ordering the carpet graphs]
\label{lem:carpet-shelling}
There is an absolute integer $c_*\ge1$ such that, for every $k$, the vertices of $\mathcal C_k$ admit an order satisfying
\[
 \card{D_r(\mathcal C_k)\cap A}\le c_*\card{\delta_{\mathcal C_k}A}
 \qquad(A\subseteq\mathcal C_k\setminus T_k)
\]
for every rank $r$.
\end{lemma}

We prove the lemma after establishing the following two geometric estimates.

\begin{lemma}[Boundary edges in standard cells]\label{lem:carpet-frontier}
There is an absolute constant $C_1$ and an order on every $\mathcal C_k$ such that the edge cut $C$ of each initial segment satisfies
\begin{equation}\label{eq:carpet-frostman}
 \#\{e\in C:e\cap Q\ne\varnothing\}\le C_1 2^\ell
 \qquad\text{for every standard $\ell$-cell }Q.
\end{equation}
\end{lemma}

\begin{proof}
At each step, replace the current axis-parallel rectangle by the smallest such rectangle containing its carpet vertices and choose a longest side. If this side has $d\ge4$ coordinate values, put
\[
 M=\left\lceil\frac{2d}{3}\right\rceil,\qquad H=2M-d+1.
\]
There are $H$ consecutive integer thresholds dividing the side into two parts of length at most $M$. Let $p=3^q$ be the largest power of three not exceeding $H$, and choose one of these thresholds with
\begin{equation}\label{eq:cantor-threshold}
 t\equiv\frac{p-1}{2}\pmod p.
\end{equation}
Thus the low $q$ ternary digits of $t$ are all $1$. Split into coordinate values at most $t$ and greater than $t$, order the first part before the second, and recurse. When both sides have at most three coordinate values, use any order.

Now consider a vertical separator. At every crossing edge, the low $q$ ternary digits of the transverse coordinate must all be $0$ or $2$. Thus at most $2^\ell$ of its crossing edges meet $Q$ when $1\le\ell\le q$. If $\ell>q$, the transverse side has length at most $d<9p$, and therefore meets only an absolute number of aligned $q$-cells, contributing $O(2^q)$. The argument for horizontal separators is the same.

Fix an initial segment and consider the successive rectangles containing its last vertex. Every edge of its boundary either crosses a line used to divide one of these rectangles or has both endpoints in the final rectangle. Along this sequence of rectangles, the side lengths at successive divisions in a fixed coordinate satisfy
\[
 d'-2\le\frac23(d-2).
\]
Since $d/9<p\le d$, only an absolute number of separators of any fixed scale $3^q$ occur in either coordinate. The separators with $q<\ell$ consequently contribute $O(\sum_{q<\ell}2^q)=O(2^\ell)$ edges incident to an $\ell$-cell.

For $\ell\ge1$, every threshold with $q\ge\ell$ has residue $(3^\ell-1)/2$ modulo $3^\ell$. Cell-boundary thresholds have residue $-1$ modulo $3^\ell$, so such a separator cannot cut an edge with just one endpoint in $Q$. In a fixed $\ell$-cell, all such thresholds in one coordinate therefore coincide, and a descendant rectangle cannot cross an already used threshold. At most one dividing line with $q\ge\ell$ in each coordinate meets the cell. Each contributes at most $2^\ell$ boundary edges. The final rectangle has bounded size, completing \eqref{eq:carpet-frostman} for $\ell\ge1$. For $\ell=0$, at most four edges meet the singleton $Q$, by the degree bound.
\end{proof}

\begin{lemma}[Relative isoperimetric inequality in a cell]\label{lem:carpet-relative}
For every standard $\ell$-cell $Q$ and every $B\subseteq Q$,
\begin{equation}\label{eq:carpet-relative-iso}
 \min\{\card B,\card{Q\setminus B}\}\le8\card{\delta_QB}^{3}.
\end{equation}
Consequently, there is an absolute $c_0>0$ such that
\begin{equation}\label{eq:carpet-relative-cheeger}
 \card{\delta_QB}\ge c_0 2^\ell
 \qquad\text{if }\card Q/16\le\card B\le\card Q/2.
\end{equation}
\end{lemma}

\begin{proof}
Color $B$ and $Q\setminus B$ differently. We want to cover the vertices of one color by disjoint standard cells $Q_i$, of levels $\ell_i$, with $\sum_i2^{\ell_i}\le2|\delta_QB|$. This will suffice, since $|Q_i|=8^{\ell_i}$ and $\sum_i8^{\ell_i}\le(\sum_i2^{\ell_i})^3$. Let us construct such a cover.

For a standard cell $R$, let $b(R)$ count its internal edges with differently colored endpoints. A $j$-cell has $2^j$ complete horizontal lines and $2^j$ complete vertical lines, obtained by requiring the transverse ternary digits to lie in $\{0,2\}$. If $b(R)<2^j$, some line of each orientation is monochromatic. Every complete horizontal line meets every complete vertical line, so all monochromatic complete lines have the same color.

Mark a $j$-cell, $j\ge1$, when $b(R)\ge2^{j-1}$. The useful observation is that, if a cell and one of its children are both unmarked, their monochromatic complete lines have the same color. To see this, note that every retained child lies in an outer row or an outer column of the $3$-by-$3$ subdivision. In the corresponding orientation all its $2^{j-1}$ complete lines extend through the parent. Fewer than $2^{j-1}$ edges change color in the parent, so at least one of these extended lines is monochromatic in both. For a level-zero child use the color of its vertex. An unmarked level-one parent has $b=0$ and is connected, so all its vertices have the same color.

If $Q$ is unmarked, every vertex of the opposite color lies in a maximal marked descendant $Q_i$, of level $\ell_i$. These cells are disjoint, and
\[
 \sum_i2^{\ell_i}\le2\sum_i b(Q_i)\le2b(Q).
\]
Thus
\[
 \min\{\card B,\card{Q\setminus B}\}
 \le\sum_i8^{\ell_i}
 \le\left(\sum_i2^{\ell_i}\right)^3
 \le8b(Q)^3.
\]
If $Q$ is marked, \eqref{eq:carpet-relative-iso} follows directly from $|Q|=8^\ell$ and $b(Q)\ge2^{\ell-1}$; for a singleton it is immediate. The density bounds in \eqref{eq:carpet-relative-cheeger} and $|Q|=8^\ell$ prove the last assertion.
\end{proof}

\begin{proof}[Proof of \Cref{lem:carpet-shelling}]
Use the order of \Cref{lem:carpet-frontier}. We first cover each $A\subseteq\mathcal C_k\setminus T_k$ by standard cells $Q_i$ of levels $\ell_i$ with
\begin{equation}\label{eq:carpet-boxing}
 \sum_i2^{\ell_i}\le C_2\card{\delta_{\mathcal C_k}A},
\end{equation}
where $C_2$ is absolute. Embed $\mathcal C_k$ as one of the eight children of a copy of $\mathcal C_{k+1}$, adding no points to $A$. Its density in the larger cell is at most $1/8$, while its density in each singleton belonging to $A$ is one. Along each descending chain ending in $A$, stop at the first cell whose density exceeds $1/2$. Its parent has density at most $1/2$ and greater than $1/(2\cdot8)$.

We keep the maximal such parents. These are disjoint, cover $A$, and satisfy \eqref{eq:carpet-relative-cheeger}. Moreover, their internal cut edges are disjoint subsets of the edge boundary of $A$. No new edge incident to $A$ was created by adding the other seven children, because $A$ avoids $T_k$. We may therefore sum the cell estimates to obtain \eqref{eq:carpet-boxing} with $C_2=c_0^{-1}$. If the larger cell was selected, we replace it by $\mathcal C_k$; this preserves the cover and decreases the sum in \eqref{eq:carpet-boxing}.

With the cover in place, fix a rank $r$ and let $C$ be the edge cut of the initial segment preceding $y_r$. From each vertex of $D_r\cap A$ preceding $y_r$, choose an edge to a vertex at rank at least $r$. The chosen edges are distinct, because each edge of $C$ has only one endpoint preceding $y_r$, and every chosen edge meets a cell of the cover. Hence
\[
 \card{D_r\cap A}\le 1+\sum_i\#\{e\in C:e\cap Q_i\ne\varnothing\}
 \le 1+C_1\sum_i2^{\ell_i}
 \le 1+C_1C_2\card{\delta_{\mathcal C_k}A}.
\]
For nonempty $A$, connectivity and $A\cap T_k=\varnothing$ give $\card{\delta_{\mathcal C_k}A}\ge1$. The empty set is immediate, so any integer $c_*\ge C_1C_2+1$ proves the lemma.
\end{proof}

We now identify these finite graphs with the level Schreier graphs, using the reflected subdivision of Matte Bon--Nekrashevych--Zheng \cite[Section~6 and Remark~6.1, p.~29]{MatteBonNekrashevychZheng2026}. The rooted tree $\{1,\ldots,8\}^*$ consists of all finite words, rooted at the empty word, with an edge between $w$ and $wx$ for each letter $x$. For a permutation $\sigma$ of the alphabet and tree automorphisms $g_1,\ldots,g_8$, wreath-recursion notation $\sigma(g_1,\ldots,g_8)$ sends $xw$ to $\sigma(x)g_x(w)$; $g_x$ is its \emph{section} at $x$, and $1$ denotes the identity automorphism. The faithful action is defined by the following recursions, with the same alphabet and generator labels.
\begin{align}\label{eq:carpet-recursion}
 a&=(12)(67)(1,1,a,1,a,1,1,a),\notag\\
 b&=(46)(58)(b,b,b,1,1,1,1,1),\notag\\
 c&=(23)(78)(c,1,1,c,1,c,1,1),\notag\\
 d&=(14)(35)(1,1,1,1,1,d,d,d).
\end{align}
To identify the level Schreier graphs, put
\[
\begin{aligned}
 p_1&=(0,0),&p_2&=(1,0),&p_3&=(2,0),&p_4&=(0,1),\\
 p_5&=(2,1),&p_6&=(0,2),&p_7&=(1,2),&p_8&=(2,2).
\end{aligned}
\]
Set $\Phi(\varnothing)=(0,0)$ for the empty word. For a word $w$ of length $m$ and $x\in\{1,\ldots,8\}$, define
\begin{equation}\label{eq:carpet-coordinate}
 \Phi(wx)=3^mp_x+R_x^{(m)}\Phi(w),
\end{equation}
where $R_x^{(m)}$ reflects the first coordinate in $[0,3^m-1]^2$ for $x\in\{2,7\}$, reflects the second for $x\in\{4,5\}$, and is the identity otherwise. Induction on $m$ shows that $\Phi$ is a bijection from the level-$m$ words to $\mathcal C_m$. The induction on edges uses the following table.
\[
\begin{array}{c|c|c|c}
 &\begin{gathered}\text{letters whose section}\\\text{equals the generator}\end{gathered}&\begin{gathered}\text{pairs interchanged}\\\text{by the generator}\end{gathered}&\begin{gathered}\text{corresponding}\\\text{side}\end{gathered}\\
\hline
 a&3,5,8&(1,2),(6,7)&\text{right}\\
 b&1,2,3&(4,6),(5,8)&\text{bottom}\\
 c&1,4,6&(2,3),(7,8)&\text{left}\\
 d&6,7,8&(1,4),(3,5)&\text{top}
\end{array}
\]
To read the table, observe that before the first moved letter, the coordinate determined by the preceding letters lies on the listed side. The indicated transposition therefore crosses the boundary between adjacent cells by one unit, and the section after that letter is trivial. The reflections in \eqref{eq:carpet-coordinate} preserve adjacency, and the four resulting edge families are disjoint. We check that these account for all the edges by counting them. For each generator, the number $M_m$ of level-$m$ words it moves satisfies
\[
 M_m=4\cdot8^{m-1}+3M_{m-1}
     =\frac{4(8^m-3^m)}5.
\]
Thus the four families contain $2M_m=8(8^m-3^m)/5$ edges in total. Writing $E_m$ for the number of nearest-neighbor edges of $\mathcal C_m$, we have $E_m=8E_{m-1}+8\cdot3^{m-1}$ and $E_1=8$, giving the same count. Every nearest-neighbor edge therefore occurs with a unique label, and the simple level-$m$ Schreier graph is $\mathcal C_m$.

\begin{proof}[Proof of \Cref{thmB}]
Choose
\[
 o_k=\left((3^k-1)/2,\,3^{k-1}-1\right)\in\mathcal C_k.
\]
Its ternary coordinates are $1^k$ and $0\,2^{k-1}$, so $d(o_k,T_k)\ge3^{k-1}-1\to\infty$. The generators are involutions, so the level graphs also define right Schreier actions. The cut inequality of \Cref{lem:carpet-shelling} verifies the hypothesis of \Cref{prop:finite-schreier}. The resulting simple Schreier graph satisfies the condition in \Cref{def:condition} with constant $c_*$. By \Cref{lem:lift}, the simple $\cS_{\mathrm{SC}}$-Cayley graph satisfies the same condition with the same constant. Thus every finite $F$ has an order with $B(P)\le(1+4c_*)B(F)$ for every initial segment $P$. This proves the edge estimate with $C_{\mathrm{SC}}=1+4c_*$, and \eqref{eq:profile-comparison} gives the vertex estimate.

The group is amenable \cite[Theorem~6.5, p.~31]{MatteBonNekrashevychZheng2026}, so \Cref{thm:record-stacking} gives the nested F{\o}lner family. Equation~\eqref{eq:word-symdiff} gives both conclusions for every finite generating set.
\end{proof}

\section{Alternating topological full groups}\label{sec:fullgroups}
We next apply the condition in \Cref{def:condition} to alternating full groups. The reason this is possible is that their orbital graphs, the Schreier graphs on individual orbits, are bilipschitz equivalent to the Cayley graph of the acting group.

Let $H$ be an infinite finitely generated group acting by homeomorphisms on a Cantor space $C$, and write the action on the left. For this action, let $\mathsf A(H,C)$ be the subgroup of its topological full group generated by the local three-cycles described in the Introduction, on triples of pairwise disjoint clopen translates $h_1U,h_2U,h_3U$.

A degree-three multisection consists of three disjoint clopen sets $U_1,U_2,U_3$ and homeomorphisms $\phi_i:U_1\to U_i$, with $\phi_1$ the identity, that locally agree with elements of $H$; the maps between components are $\phi_j\phi_i^{-1}$. Equivalently, $\mathsf A(H,C)$ is generated by the cycles of these multisections. Indeed, refine a multisection into clopen pieces on which the maps from the first component are translations by fixed elements of $H$. Its cycle is then a product of the local cycles above. The converse is immediate. An action on a Cantor space is \emph{expansive} if the translates of some finite clopen partition distinguish every pair of distinct points. For a free minimal action the group is simple; if the action is also expansive, it is finitely generated \cite[Theorems~1.1, 1.2 and~5.10]{Nekrashevych2017}.

\begin{lemma}[Orbital graphs]\label{lem:orbital}
Let an infinite finitely generated group $H$ act freely on $C$, and suppose $\Lambda=\mathsf A(H,C)$ is finitely generated. Each $H$-orbit is a single $\Lambda$-orbit. With finite symmetric generating sets fixed, its $\Lambda$-Schreier graph is bilipschitz equivalent to the Cayley graph of $H$, with constants uniform over the orbits.
\end{lemma}

A more general result appears in \cite[Lemma~5.14]{MatteBon}. We give the proof with our convention of right actions.

\begin{proof}
Let $S$ and $T$ generate $H$ and $\Lambda$, respectively. We will compare the two metrics by using bounded piecewise translations in one direction and finitely many local three-cycles in the other. Although the homeomorphisms act on the left, we use the right actions
\[
 y\cdot h=h^{-1}y,\qquad y\cdot\lambda=\lambda^{-1}y.
\]
Fix $z_0\in C$. Freeness gives a bijection $\theta:H\to Hz_0$ defined by $\theta(h)=h^{-1}z_0$, and
\[
 \theta(hs)=s^{-1}\theta(h)=\theta(h)\cdot s.
\]
We give the orbit the metric transported from $H$ by $\theta$.

For each $t\in T$, there is a finite clopen partition on whose pieces $t^{-1}$ has the form $y\mapsto ky$. Let $L$ bound the $S$-lengths of the finitely many elements used in these partitions. At $y=\theta(h)$, a right $T$-edge ends at
\[
 y\cdot t=t^{-1}y=kh^{-1}z_0=\theta(hk^{-1}).
\]
This edge has $H$-length $|k|_S\le L$. Hence $d_H\le Ld_T$ on each $\Lambda$-orbit.

For the reverse bound, fix $s\in S$ and $z\in C$, and choose $k\in H\setminus\{e,s^{-1}\}$. Freeness makes $z,s^{-1}z,kz$ distinct. Choose a clopen neighborhood $U$ of $z$ with $U,s^{-1}U,kU$ pairwise disjoint. Its local cycle $\tau_U$ belongs to $\Lambda$ and satisfies $\tau_U(y)=s^{-1}y$ for $y\in U$. For each $s$, compactness gives finitely many such neighborhoods covering $C$. Let $M$ bound the $T$-word lengths of their cycles. For every $y\in C$, one of these cycles satisfies
\[
 y\cdot\tau_U^{-1}=\tau_U(y)=s^{-1}y=y\cdot s.
\]
Since $T$ is symmetric, $\card{\tau_U^{-1}}_T\le M$. Thus every $H$-edge can be replaced by a $T$-path of length at most $M$. The orbits coincide and $d_T\le Md_H$, uniformly over the orbits.
\end{proof}

\begin{theorem}[Profiles of alternating full groups]\label{thm:fullgroup}
Let $H$ be infinite and finitely generated, and suppose its Cayley graph satisfies the condition in \Cref{def:condition}. Let $H$ act freely on a Cantor space $C$, and assume $\Lambda=\mathsf A(H,C)$ is finitely generated. Then, for every finite symmetric generating set $T$ of $\Lambda$, there are $K,C_0<\infty$ and $n_0$ such that
\[
 \IexE^{\Lambda,T}(n)\le K\IupE^{\Lambda,T}(n)\qquad(n\ge1)
\]
and there are nested sets $W_n$, $n\ge n_0$, satisfying
\[
 \card{W_n}=n,\qquad \bigcup_nW_n=\Lambda,\qquad
 B_T(W_n)\le C_0\IupE^{\Lambda,T}(n).
\]
Every finite $F\subset\Lambda$ also admits an order with $B_T(P)\le K B_T(F)$ for each initial segment $P$. If $\Lambda$ is amenable, $(W_n)$ is F\o lner. For outer vertex profiles, $I^{\Lambda,T}(n)\le\card T K I_\uparrow^{\Lambda,T}(n)$.
\end{theorem}

\begin{proof}
By \Cref{lem:orbital}, we have an infinite transitive $\Lambda$-Schreier graph bilipschitz equivalent to the Cayley graph of $H$, so \Cref{thm:schreier} gives the initial-segment and profile estimates and the nested exhaustion. The F{\o}lner conclusion follows from \Cref{thm:record-stacking}, and the vertex estimate follows from \eqref{eq:profile-comparison}.
\end{proof}

\begin{proposition}[Nonexistence of height functions]\label{prop:noheight}
Let $\Z^d$, $d\ge2$, act freely and minimally on a Cantor space $C$, and suppose $\Lambda=\mathsf A(\Z^d,C)$ is finitely generated. Then $\Lambda$ has no infinite transitive action with a height function satisfying \eqref{eq:schreier-height-data}. The same holds for every nontrivial quotient of a finite-index subgroup of $\Lambda$.
\end{proposition}

\begin{proof}
The action groupoid records arrows $x\to gx$ for $x\in C$ and $g\in\Z^d$, composed by successive applications of group elements. It is compactly generated by the chosen generators over the compact space $C$, minimal because every orbit is dense, and effective because no nonidentity group element fixes a nonempty open set. Its orbital growth, the number of orbit points reachable in at most $r$ generator steps, is comparable to $r^d$. By Matte Bon's Theorem~8.5 and Proposition~8.1(iv) \cite{MatteBon}, for every infinite transitive $\Lambda$-Schreier graph $X$ there are constants $c_X>0$ and $r_X\ge1$ such that $\sup_{x\in X}\card{B_X(x,r)}\ge c_Xr^d$ for every integer $r\ge r_X$. The constants depend on $X$ and the generating set.

A height function satisfying \eqref{eq:schreier-height-data} would give, for every vertex $x$,
\[
 \card{B_X(x,r)}\le W(2Lr+1).
\]
Indeed, $h(B_X(x,r))\subset[h(x)-Lr,h(x)+Lr]$. This contradicts the lower bound when $d\ge2$.

Since $\Lambda$ is infinite and simple, it has no proper finite-index subgroups and no nontrivial proper quotients. This proves the last assertion.
\end{proof}

\subsection{Examples}
For a group $H$ and a finite alphabet $\mathcal A$, an $H$-subshift is a closed $H$-invariant subset of $\mathcal A^H$, with the product topology for the discrete alphabet $\mathcal A$ and with $H$ acting by shifting coordinates.

The condition in \Cref{def:condition} holds with $c=W(L+1)$ under the hypotheses of \Cref{thm:schreier-line-ratio}, and with the absolute constant $c_*$ for the carpet group by \Cref{lem:carpet-shelling,prop:finite-schreier}. It passes from finite-index subgroups and quotients to the group by \Cref{thm:transport,lem:lift}. Thus \Cref{thm:fullgroup} applies when $H$ is infinite, finitely generated and elementary amenable, when $H$ is a Grigorchuk group or the carpet group, and in the other cases covered by \Cref{thm:schreier-line-ratio}. For these $H$ one may take a free minimal subshift on a Cantor space, which exists by \cite[Theorem~1.4.1, p.~6]{GaoJacksonSeward2012}; its alternating full group is finitely generated by \cite[Proposition~5.7 and Theorem~5.10, pp.~720--721]{Nekrashevych2017}. The F{\o}lner conclusion requires amenability of the full group.

\begin{corollary}[Subshifts and interval exchanges]\label{cor:abelian-fullgroups}
For every $d\ge1$, \Cref{thm:fullgroup} applies to the derived full group of every faithful minimal $\Z^d$-subshift and to the derived interval-exchange groups obtained by splitting the circle along a rotation subgroup isomorphic to $\Z^d$. For $d\ge2$, these groups admit no infinite transitive action with a height function satisfying \eqref{eq:schreier-height-data}.
\end{corollary}

\begin{proof}
A faithful minimal action of an abelian group is free \cite[Lemma~3]{CJN}. For a subshift, the action is also expansive, and \cite[Theorem~1 and Proposition~8]{CJN} gives finite generation of the derived group and identifies it with the alternating group. For interval exchanges, we use \cite[Theorem~12]{CJN}. Finally, projection of $\Z^d$ onto $\Z$ gives the condition in \Cref{def:condition}, so we can apply \Cref{thm:fullgroup} and \Cref{prop:noheight}.
\end{proof}

\begin{corollary}[Derived groups of interval exchanges of rank two]\leavevmode\phantomsection\label{cor:iet}\ignorespaces
Let $H\le\R/\Z$ be a group of rotations isomorphic to $\Z^2$, generated by two irrational rotations. Let $C_H$ be the Cantor space obtained by splitting each point of $H$ into its left and right copies, and let $\operatorname{IET}_H$ be the topological full group of its rotation action. Then
\[
 \Lambda=\operatorname{IET}_H'=\mathsf A(H,C_H)
\]
is infinite, finitely generated, simple and amenable. For every finite generating set it satisfies the conclusions of \Cref{thm:fullgroup}, including a nested F{\o}lner exhaustion. It admits no infinite transitive action with a height function satisfying \eqref{eq:schreier-height-data}.

One explicit choice is $H=\langle\sqrt2+\Z,\sqrt3+\Z\rangle$.
\end{corollary}

\begin{proof}
The split-circle construction and its identification with interval exchanges are described in \cite[Section~3]{CJN}. The rotation action is free and minimal. The proof of \Cref{cor:abelian-fullgroups} gives $\operatorname{IET}_H'=\mathsf A(H,C_H)$, finite generation, the conclusions of \Cref{thm:fullgroup}, and the nonexistence of an infinite transitive action with a height function satisfying \eqref{eq:schreier-height-data}. Simplicity follows from Theorem~12 of that paper; infinitude follows from arbitrarily many disjoint local three-cycles.

Every translation increment of $\operatorname{IET}_H$ belongs to $H$, so its angle group, the subgroup of $\R/\Z$ generated by these increments, lies in $H\cong\Z^2$. Its rational rank, the largest $d$ for which it contains a subgroup isomorphic to $\Z^d$, is therefore at most two. Amenability follows from \cite[Theorem~1.9]{JMBMdlS}, making the nested exhaustion F{\o}lner. Finally, $1,\sqrt2,\sqrt3$ are linearly independent over $\mathbb Q$, so the displayed rotation group is isomorphic to $\Z^2$.
\end{proof}

\section{Nested sets in polycyclic groups}\label{sec:polycyclic}

The height-function argument already gives a nested exhaustion for polycyclic groups. We now refine Pittet's boxes \cite[pp.~681--684]{Pittet1995} to obtain an explicit exhaustion with sets of cardinality $n$ and boundary comparable to $n/\log n$ in exponential growth.

\begin{proposition}[Nested sets in polycyclic groups]\label{thmC}
Let $G$ be a polycyclic group of exponential growth and let $\cS$ be a finite symmetric generating set. There is an explicit nested exhaustion $(F_n)_{n\ge n_0}$ such that
\[
 |F_n|=n,\qquad |\partial_{\cS}F_n|\asymp\frac n{\log n}\qquad(n\ge n_0).
\]
Consequently $I(n)\asymp I_\uparrow(n)\asymp n/\log n$ for all sufficiently large $n$. The sets $F_n$ are obtained by interpolating an explicit nested exhaustion $(E_R)$ satisfying $|\partial_{\cS}E_R|\asymp|E_R|/\log|E_R|$ and $|E_{R+1}\setminus E_R|=O(|E_R|/\log|E_R|)$.
\end{proposition}

For a polycyclic group of exponential growth, successive boxes in Pittet's family differ by a set comparable in cardinality to the smaller box \cite{Pittet1995}. We therefore choose different radii in the nilpotent subgroup and the abelian quotient. This gives a nested exhaustion $(\Omega_R)$ with $|\Omega_{R+1}\setminus\Omega_R|=O(|\partial\Omega_R|)$, which is the bound used for the interpolation below.

\subsection{Normal form and displacement}

\begin{proposition}[Normal form for a finite-index subgroup]\label{prop:pittet-normal-form}
Let $G$ be an infinite polycyclic group. A finite-index subgroup $H$ fits into an exact sequence $1\to N\to H\to\Z^r\to1$ with $N$ finitely generated and nilpotent. There are a finite symmetric generating set $\mathcal B$ of $N$, lifts $a_1,\ldots,a_r$ of the standard basis of $\Z^r$, and $q>1$ such that every $g\in H$ has a unique representation
\[
  g=x a_1^{k_1}\cdots a_r^{k_r},
  \qquad x\in N,\quad k\in\Z^r,
\]
and, with $\card{k}_1=\sum_i\card{k_i}$, right multiplication by any generator in
\[
  \cT:=\mathcal B\cup\{a_1^{\pm1},\ldots,a_r^{\pm1}\}
\]
replaces $x$ by $xu$, where $\card{u}_{\mathcal B}\le q^{|k|_1}$, and replaces $k$ by $k$ or $k\pm e_i$ $(1\le i\le r)$.
\end{proposition}

\begin{proof}
We use Pittet's finite-index reduction \cite[pp.~678--681]{Pittet1995}. The image in $\Z^r$ determines $k$, after which $x$ is unique. To bound the displacement in the nilpotent coordinate, let $M\ge1$ bound the Lipschitz constants of conjugation by the $a_i^{\pm1}$ on $(N,\card{\cdot}_{\mathcal B})$, and put $\ell=|k|_1$. Multiplication by $b\in\mathcal B$ changes $x$ to $x(a_1^{k_1}\cdots a_r^{k_r})b(a_1^{k_1}\cdots a_r^{k_r})^{-1}$, a displacement of length at most $M^\ell$. To multiply by $a_i^{\pm1}$, move it past at most $\ell$ letters. Each interchange contributes a conjugate of one of finitely many commutators, by a word of length at most $\ell+1$. The displacement is therefore at most $C\ell M^{\ell+1}\le q^\ell$ after increasing $q$; for $\ell=0$ it is zero. The quotient coordinate changes as stated. Exponential growth forces $r\ge1$, since $r=0$ gives $H=N$.
\end{proof}

Let
\[
  V(R):=\card{\{x\in N:\card{x}_{\mathcal B}\le R\}},
  \qquad
  P_r(T):=\card{\{k\in\Z^r:\card{k}_1\le T\}}.
\]
The volume of balls in $N$ satisfies, for some $D\in\N$, $v_0>0$ and $\eta>0$,
\begin{equation}\label{eq:nilpotent-growth-error}
  V(R)=v_0R^D+O(R^{D-\eta}).
\end{equation}
Breuillard--Le Donne prove \eqref{eq:nilpotent-growth-error} with a positive error exponent depending only on the nilpotency class \cite{BreuillardLeDonne2013}; Pittet used an earlier power-saving estimate \cite[pp.~681--683]{Pittet1995}. If $G$ has exponential growth, $N$ is infinite and $D\ge1$: otherwise $H$ would be finite-by-$\Z^r$, hence virtually abelian. Also
\begin{equation}\label{eq:l1-ball-growth}
  P_r(T)=p_rT^r+O(T^{r-1})
\end{equation}
for a constant $p_r>0$.

\subsection{The boxes}

How should we choose the quotient radius $T_R$? It must tend to infinity, but the nilpotent displacement $q^{T_R}$ must remain sublinear in $R$. A logarithmic choice has both properties. Moreover, when this radius increases by one, the added points form only an $O(1/T_R)$ fraction of the box. We therefore fix
\[
  0<\alpha<\frac1{\log q},
  \qquad T_R=\lfloor\alpha\log R\rfloor,
\]
and put
\begin{equation}\label{eq:polybox}
  \Omega_R:=
  \left\{x a_1^{k_1}\cdots a_r^{k_r}:
  \card{x}_{\mathcal B}\le R,\ \card{k}_1\le T_R\right\}\subset H.
\end{equation}

\begin{proposition}[Cardinality and boundary of the boxes]\label{prop:polybox-estimates}
There is $C<\infty$ such that, for all sufficiently large integers $R$,
\begin{align}
  \card{\Omega_R}&\asymp R^D(\log R)^r,\label{eq:polybox-volume}\\
  B_{\cT}(\Omega_R)&\le C\frac{\card{\Omega_R}}{\log R},\label{eq:polybox-boundary}\\
  \card{\Omega_{R+1}\setminus\Omega_R}&\le C\frac{\card{\Omega_R}}{\log R}.\label{eq:polybox-gap}
\end{align}
The sets $\Omega_R$ are nested.
\end{proposition}

\begin{proof}
Uniqueness of the normal form gives
\[
  \card{\Omega_R}=V(R)P_r(T_R),
\]
which gives \eqref{eq:polybox-volume} by \eqref{eq:nilpotent-growth-error}--\eqref{eq:l1-ball-growth}. The sets are nested because both radii are nondecreasing.

To estimate the boundary, write $T=T_R$, $s=q^T$, and $a^k=a_1^{k_1}\cdots a_r^{k_r}$. Since $\beta:=\alpha\log q<1$, we have $s\le R^\beta=o(R)$. We remove a shell of width $s$ in the nilpotent coordinate and one layer in the quotient, leaving
\[
  \omega_R:=
  \left\{x a^k:\card{x}_{\mathcal B}\le R-s,\ \card{k}_1\le T-1\right\}.
\]
By \Cref{prop:pittet-normal-form}, every right $\cT$-neighbor of a point of $\omega_R$ remains in $\Omega_R$. Hence every crossing edge starts in $\Omega_R\setminus\omega_R$, and therefore
\[
  B_{\cT}(\Omega_R)
  \le \card{\cT}\bigl(V(R)P_r(T)-V(R-s)P_r(T-1)\bigr).
\]
The growth expansions imply
\[
  \frac{V(R)-V(R-s)}{V(R)}
  =O\left(\frac{s}{R}+R^{-\eta}\right),
  \qquad
  \frac{P_r(T)-P_r(T-1)}{P_r(T)}=O(T^{-1}).
\]
Since $R^{\beta-1}+R^{-\eta}=o((\log R)^{-1})$ and $T\asymp\log R$, this proves \eqref{eq:polybox-boundary}.

It remains to estimate the difference between successive boxes. For large $R$, we have $T_{R+1}-T_R\in\{0,1\}$. If the quotient radius does not increase, then
\[
  \card{\Omega_{R+1}\setminus\Omega_R}
  =(V(R+1)-V(R))P_r(T).
\]
From \eqref{eq:nilpotent-growth-error},
\[
  V(R+1)-V(R)=O(R^{D-1}+R^{D-\eta}),
\]
which is $O(R^DT^{-1})$. If $T_{R+1}=T+1$, there is the additional term
\[
  V(R+1)(P_r(T+1)-P_r(T))=O(R^DT^{r-1}).
\]
Thus both cases give \eqref{eq:polybox-gap}, as required.
\end{proof}

\subsection{Interpolation and finite index}

\begin{lemma}[Interpolation]\label{lem:checkpoint-interpolation}
Let $E_1\subset E_2\subset\cdots$ be nonempty finite sets, with $N_k=\card{E_k}\to\infty$ and $b_k=\card{\bd E_k}$. There is a nested family $(F_n)_{n\ge N_1}$ such that $\card{F_n}=n$, $F_{N_k}=E_k$, $\bigcup_nF_n=\bigcup_kE_k$, and
\begin{equation}\label{eq:checkpoint-interpolation}
  \card{\bd F_n}\le b_k+\Delta(n-N_k)\qquad(N_k\le n\le N_{k+1}).
\end{equation}
\end{lemma}

\begin{proof}
Starting from $E_1$, we add the vertices of each difference $E_{k+1}\setminus E_k$ one at a time, in any order. This gives the stated cardinalities and union, while \Cref{lem:boundary-editing} gives \eqref{eq:checkpoint-interpolation}.
\end{proof}

\begin{proof}[Proof of \Cref{thmC}]
Take the finite-index subgroup $H$ and boxes $\Omega_R$ above, choose a finite right transversal $Q$ for $H$ in $G$, and put $E_R=\Omega_RQ$, $N_R=\card{E_R}$. By \Cref{prop:polybox-estimates,lem:finite-index-thickening}, these sets are nested and
\[
  N_R\asymp R^D(\log R)^r,
  \qquad
  \card{\bd E_R}+\card{E_{R+1}\setminus E_R}
  \lesssim\frac{N_R}{\log N_R}.
\]
Since $R,T_R\to\infty$, the normal form gives $\bigcup_RE_R=G$. The lower bound in \Cref{prop:exp-profile-lower} gives $\card{\bd E_R}\gtrsim N_R/\log N_R$, as well as $\card{\bd E_R}\lesssim\Iup(N_R)$. The interpolation bound in \Cref{lem:checkpoint-interpolation} now gives a nested exhaustion $(F_n)_{n\ge n_0}$ with $|F_n|=n$ and $\card{\bd F_n}\lesssim n/\log n$, since $n/\log n$ is eventually increasing. The lower bounds for $\Iex(n)$ and $\Iup(n)$ in \Cref{prop:exp-profile-lower} complete the estimates for the nested sets and the profiles.
\end{proof}

Every infinite polycyclic group is amenable and virtually indicable. Thus, without a growth assumption, \Cref{cor:virtually-indicable} still gives a nested F{\o}lner exhaustion $(F_n)$ with $|F_n|=n$ and $|\partial F_n|=O(I_\uparrow(n))$.

\section{Further consequences for profiles}\label{sec:profile-consequences}

The normalized estimates in \Cref{sec:atoms} also give bounds for separation and Dirichlet profiles. We prove these bounds first, and then give the counterexample to the inverse-composition condition in Correia's question.

\subsection{Separation and Dirichlet profiles}

Let $X$ be an infinite connected vertex-transitive graph of finite degree $\Delta$, with profiles $a,j$ as in \Cref{sec:atoms}. We use the relative boundary $B_A(U)$ defined before \Cref{lem:relative-averaging}.

For $|A|\ge2$, put
\[
 h(A)=\min_{\varnothing\ne U\subset A,\ |U|\le |A|/2}\frac{B_A(U)}{|U|},\qquad
 \operatorname{Sep}_X(n)=\sup_{2\le |A|\le n}|A|h(A),
\]
and set $\operatorname{Sep}_X(1)=0$. These are the edge-boundary conventions of \cite[Definitions~2.2--2.4]{GournayLeCoz2023}, up to a fixed normalization.

\begin{corollary}[Separation profile]\leavevmode\phantomsection\label{cor:separation-profile}\ignorespaces
For every infinite connected vertex-transitive graph $X$ of finite degree,
\[
 \operatorname{Sep}_X(n)\le C\frac n{j(n)}\qquad(n\ge1).
\]
\end{corollary}
\begin{proof}
Suppose first that $X$ is amenable. Given $A$ with $2\le m=|A|\le n$, put $k=\lfloor m/2\rfloor$ and choose $F$ attaining $j(k)$. \Cref{lem:relative-averaging} gives a nonempty $U\subset A$ with $|U|\le k$ and $B_A(U)/|U|\le1/j(k)$. By \eqref{norm:eq:linear},
\[
 |A|h(A)\le\frac m{j(k)}
 \le\Delta\frac mk\frac n{j(n)}
 \le3\Delta\frac n{j(n)}.
\]
Take the supremum over $A$. If $X$ is nonamenable, $j$ is bounded and $\operatorname{Sep}_X(n)\le\Delta n$, which proves the result.
\end{proof}

Since the normalized boundary profile of \cite{GournayLeCoz2023} is comparable to $1/j$, \Cref{cor:separation-profile} answers their Question~7.3(1) at the same cardinality.

We next use Hume's max-gradient convention \cite{Hume2019}. For a finite induced subgraph $A$, write
\[
 \|\nabla_A f\|_1=\sum_{x\in A}\max_{y\in A,\,y\sim x}|f(x)-f(y)|,
 \qquad \overline f_A=|A|^{-1}\sum_{x\in A}f(x),
\]
where the maximum is zero at an isolated vertex. Let $\partial_X^{\mathrm{in}}A$ be the set of vertices of $A$ adjacent to $X\setminus A$. The Neumann and Dirichlet constants are
\[
 \begin{split}
 h_N^1(A)&=\inf_{f\text{ nonconstant}}\frac{\|\nabla_A f\|_1}{\|f-\overline f_A\|_1},\\
 h_D^1(A)&=\inf_{0\ne f,\ f|_{\partial_X^{\mathrm{in}}A}=0}\frac{\|\nabla_A f\|_1}{\|f\|_1}.
 \end{split}
\]
The infimum over an empty class in the second definition is $+\infty$. Define the Neumann and Dirichlet profiles, respectively, by
\[
 \Lambda_X^1(n)=\sup_{2\le|A|\le n}|A|h_N^1(A),\qquad
 D\Lambda_X^1(n)=\inf_{|A|\ge n}|A|h_D^1(A),
\]
with $\Lambda_X^1(1)=0$.

\begin{corollary}[Dirichlet and Neumann profiles]\label{cor:dirichlet-neumann}
For every infinite connected vertex-transitive graph of finite degree,
\[
 D\Lambda_X^1(n)\asymp\frac n{j(n)},\qquad
 \Lambda_X^1(n)\le C D\Lambda_X^1(n).
\]
In particular, Hume's Question~1.1 has an affirmative answer at $p=1$ for finitely generated groups.
\end{corollary}
\begin{proof}
If $|A|=m$ and $f$ vanishes on $\partial_X^{\mathrm{in}}A$, extend $f$ by zero to $X$. No additional gradient term appears outside $A$. The degree bound and edge coarea applied to $|f|$ give
\[
 \|\nabla_A f\|_1\ge\frac2\Delta\sum_{\{x,y\}\in E_X}|f(x)-f(y)|
 \ge\frac{2\|f\|_1}{\Delta j(m)}.
\]
For $m\ge n$, \eqref{norm:eq:linear} gives $m/j(m)\ge n/(\Delta j(n))$. Hence $D\Lambda_X^1(n)\ge2n/(\Delta^2j(n))$.

For the reverse bound, choose $F$ attaining $j(n)$ and let $A$ be the closed one-neighborhood of $F$. Add vertices if necessary so that $n\le|A|\le(\Delta+1)n$. The function $\one_F$ vanishes on $\partial_X^{\mathrm{in}}A$ and has $\|\nabla_A\one_F\|_1\le2B(F)$. Thus $D\Lambda_X^1(n)\le2(\Delta+1)n/j(n)$.

Finally, for $U\subset A$ with $0<|U|\le|A|/2$,
\[
 \|\one_U-|U|/|A|\|_1=2|U|(1-|U|/|A|)\ge|U|,
 \qquad \|\nabla_A\one_U\|_1\le2B_A(U).
\]
It follows that $\Lambda_X^1(n)\le2\operatorname{Sep}_X(n)$, so \Cref{cor:separation-profile} and the Dirichlet lower bound complete the proof.
\end{proof}

\subsection{Inverse composition of profiles}\label{sec:inverse-composition}

The first part of Correia's Question~1.3 asks for fixed-dilation regularity of a single normalized profile, which follows from \Cref{norm:cor:lone}. The second part asks whether, when $j_G=o(j_H)$, equivalent representatives can always be chosen to satisfy the additional regularity condition on their inverse composition. The following construction gives a negative answer.

For a group with its word metric, \emph{asymptotic dimension one} means that, for every $R>0$, it can be covered by two families of uniformly bounded sets such that distinct sets in either family are at distance greater than $R$. A single family does not suffice at every scale.

\begin{proposition}\label{prop:inverse-composition-counterexample}
There are infinite finitely generated elementary amenable groups $G,H$ of asymptotic dimension one such that $j_G=o(j_H)$, but no pair of representatives $h_G\simeq j_G$ and $h_H\simeq j_H$, with $h_G$ nondecreasing, $h_H$ strictly increasing, and $h_H^{-1}$ defined for all sufficiently large arguments, satisfies
\[
 (h_G\circ h_H^{-1})(Ct)=O_C\bigl((h_G\circ h_H^{-1})(t)\bigr)
 \qquad\text{for every }C>0.
\]
\end{proposition}
\begin{proof}
We use the prescribed-profile theorem of Brieussel--Zheng \cite[Theorem~1.1]{BrieusselZheng2021}, in the form recorded in \cite[Proposition~2.14]{Escalier2024}. It realizes $j(v)\asymp\rho(\log v)$ whenever $\rho:[1,\infty)\to[1,\infty)$ is continuous, $\rho(1)=1$, and both $\rho(t)$ and $t/\rho(t)$ are nondecreasing. At $p=1$, this follows by taking $f(t)=t/\rho(t)$ in their Theorem~1.1 and reciprocating its decreasing profile; edge coarea and the use of a measure whose finite support generates the group change the normalization only by fixed constants.

We choose $\rho_H$ to be constant on intervals where $\rho_G$ increases by arbitrarily large factors, while still arranging $\rho_G=o(\rho_H)$. After the logarithmic change of variable, every equivalent representative will be pointwise comparable to $\rho_i(\log v)$. At the exponentiated interval endpoints, the ratios remain bounded for $h_H$ and tend to infinity for $h_G$. A fixed dilation in the range of $h_H$ therefore gives an unbounded multiplicative increase in $h_G\circ h_H^{-1}$.

Put $a_k=(k!)^2$ for $k\ge2$, and define $\rho_G(t)=\sqrt t$. Define $\rho_H(t)=t$ on $[1,4]$ and
\[
 \rho_H(t)=\max\{k k!,t/k!\}\qquad(a_k\le t\le a_{k+1},\ k\ge2).
\]
The defining formulas agree at their common endpoints, since $\rho_H(a_k)=k k!$. On each interval, both $\rho_H$ and
\[
 \frac t{\rho_H(t)}=\min\{t/(k k!),k!\}
\]
are nondecreasing. Thus $\rho_H$ is continuous and satisfies the realization hypotheses, as does $\rho_G$. Both are unbounded. Moreover,
\[
 \frac{\rho_H(t)}{\sqrt t}
 =\max\left\{\frac{k k!}{\sqrt t},\frac{\sqrt t}{k!}\right\}\ge\sqrt k
 \qquad(a_k\le t\le a_{k+1}),
\]
so $\rho_G=o(\rho_H)$. Put $b_k=ka_k$. Then
\[
 \rho_H(a_k)=\rho_H(b_k)=k k!,\qquad
 \frac{\rho_G(b_k)}{\rho_G(a_k)}=\sqrt k.
\]
The realization theorem gives groups with $j_G(v)\asymp f_G(v)$ and $j_H(v)\asymp f_H(v)$, where $f_i(v)=\rho_i(\log v)$. Since $\rho_i$ is unbounded, the corresponding functions $f(t)=t/\rho_i(t)$ satisfy $f(t)/t\to0$, so the elementary amenable, asymptotic-dimension-one case of \cite[Theorem~1.1]{BrieusselZheng2021} applies. The preceding estimate gives $j_G=o(j_H)$.

For either function $\rho_i$, its two monotonicity properties imply
\[
 \rho_i(t)\le\rho_i(ct)\le c\rho_i(t)\qquad(c\ge1).
\]
It follows that $\rho_i(t+L)/\rho_i(t)\to1$ for each fixed real $L$, and hence $f_i(Cv)/f_i(v)\to1$ for every fixed $C>0$. Consequently every representative $h_i\simeq j_i$ is pointwise comparable to $f_i$: for some $A_i\ge1$,
\[
 A_i^{-1}f_i(v)\le h_i(v)\le A_i f_i(v)
\]
for all sufficiently large $v$.

Put $x_k=\exp(a_k)$, $y_k=\exp(b_k)$ and $F=h_G\circ h_H^{-1}$. Then
\[
 \frac{h_H(y_k)}{h_H(x_k)}\le A_H^2,
 \qquad
 \frac{h_G(y_k)}{h_G(x_k)}\ge\frac{\sqrt k}{A_G^2}.
\]
Choose a fixed $C>A_H^2$ and set $u_k=h_H(x_k)$. Since $F$ is nondecreasing and $u_k\to\infty$,
\[
 \frac{F(Cu_k)}{F(u_k)}
 \ge\frac{F(h_H(y_k))}{F(h_H(x_k))}
 =\frac{h_G(y_k)}{h_G(x_k)}\longrightarrow\infty.
\]
This contradicts the proposed dilation condition.
\end{proof}

\section{Left and right boundaries}\label{sec:autocorrelation}
For a finite set $F\subset\GammaG$, put $B_\Sigma(F)=\edgebd(F)+B^{\mathrm L}_{\cS}(F)$, and define
\[
 \IexE^\Sigma(n)=\min_{\card F=n}B_\Sigma(F),
 \qquad \IupE^\Sigma(n)=\inf_{\card F\ge n}B_\Sigma(F).
\]
These are the edge profiles for the left--right action considered by Gromov \cite[p.~540]{Gromov2008}. Inversion gives $B^{\mathrm L}_{\cS}(F)=\edgebd(F^{-1})$, so $B_\Sigma(F^{-1})=B_\Sigma(F)$, $\IexE^\Sigma\ge2\IexE$, and $\IupE^\Sigma\ge2\IupE$.

\begin{proposition}[Comparison with the left--right profile]\label{cor:left-right-comparison}
For every $n\ge1$,
\begin{equation}\label{eq:left-right-comparison}
 \IexE(n)\le(\Delta+1)\IupE^\Sigma(n).
\end{equation}
Consequently, if $\IupE^\Sigma(n)\le C\IupE(n)$ for every $n$, then $\IexE(n)\le C(\Delta+1)\IupE(n)$, and $\Iex/\Iup$ is bounded.
\end{proposition}
\begin{proof}
Choose $F$ with $\card F\ge n$ and $B_\Sigma(F)=\IupE^\Sigma(n)$. Put $R=\edgebd(F)$ and $L_0=B^{\mathrm L}_{\cS}(F)$. We compare $F$ with successive left translates until the intersection is empty. The right cut bounds the boundary of each intersection, while the left cut bounds the change in its cardinality at each step. Replacing $F$ by $F^{-1}$ if necessary gives $L_0\le R$. Choose $g\notin FF^{-1}$, write $g=s_1\cdots s_\ell$, and put $g_0=e$, $g_i=s_1\cdots s_i$. The intersections $F\cap g_iF$ have right edge cut at most $2R$, since left translations preserve this cut. Their cardinalities start at $\card F$ and end at zero, and consecutive cardinalities differ by at most
\[
 \card{g_{i+1}F\triangle g_iF}
 =\card{s_{i+1}F\triangle F}\le2L_0.
\]
Taking the first intersection of cardinality at most $n$ and adding at most $2L_0$ points gives an $n$-point set with right edge cut at most $2R+2\Delta L_0$. Therefore
\[
 \IexE(n)\le2R+2\Delta L_0
 \le(\Delta+1)(R+L_0)
 =(\Delta+1)\IupE^\Sigma(n).
\]
The final assertion follows from \eqref{eq:profile-comparison}.
\end{proof}

\begin{question}\label{ques:left-right-profiles}
For every infinite finitely generated group $\GammaG$ and finite symmetric generating set $\cS$, is there $C<\infty$ such that
\[
 \IupE^\Sigma(n)\le C\IupE(n)\qquad(n\ge1)?
\]
\end{question}

\subsection*{Acknowledgments} The author thanks the Mathematics Department at IIT Bombay for providing ideal working conditions.

\end{document}